\documentclass[12pt]{amsart}
\usepackage{amssymb}
\usepackage{amsthm}
\usepackage{amsmath}
\usepackage[dvips]{epsfig}
\usepackage{graphicx}
\usepackage{color}
\usepackage{url}
\usepackage{tikz}
\usepackage{caption}
\usepackage{subcaption}
\usepackage{epstopdf}
\usepackage[arrow, matrix, curve]{xy}

\usetikzlibrary{decorations.markings}

\makeatletter\@addtoreset {equation}{section}\makeatother

\theoremstyle{plain}
\newtheorem{definition}{Definition}[section]

\newtheorem{theorem}[definition]{Theorem}
\newtheorem{corollary}[definition]{Corollary}
\newtheorem{lemma}[definition]{Lemma}

\newtheorem{remark}[definition]{Remark}

\DeclareMathOperator{\sech}{sech}

\newenvironment{proof1}%
{\begin{trivlist} \item[]{\em Proof }}%
{\hspace*{\fill}$\rule{.3\baselineskip}{.35\baselineskip}$\end{trivlist}}

\begin{document}

\title[Nonlinear Instability of Half-Solitons on Star Graphs]
{Nonlinear Instability of Half-Solitons on Star Graphs}

\author{Adilbek Kairzhan}
\address{Department of Mathematics, McMaster University, Hamilton, Ontario  L8S 4K1, Canada}
\email{kairzhaa@math.mcmaster.ca}

\author{Dmitry Pelinovsky}
\address{Department of Mathematics, McMaster University, Hamilton, Ontario  L8S 4K1, Canada}
\email{dmpeli@math.mcmaster.ca}

\thanks{This research was supported by the NSERC Discovery Grant.}
\date{\today}

\begin{abstract}
We consider a half-soliton stationary state of the nonlinear Schr\"{o}dinger equation with the power nonlinearity
on a star graph consisting of $N$ edges and a single vertex. For the subcritical power nonlinearity,
the half-soliton state is a degenerate critical point of the action functional under the mass constraint such that
the second variation is nonnegative. By using normal forms, we prove that the degenerate critical point
is a nonlinear saddle point, for which the small perturbations to the half-soliton state grow slowly in time resulting
in the nonlinear instability of the half-soliton state. The result holds for any $N \geq 3$ and
arbitrary subcritical power nonlinearity. It gives a precise dynamical characterization of the previous result
of Adami {\em et al.}, where the half-soliton state was shown to be a saddle point of the
action functional under the mass constraint for $N = 3$ and for cubic nonlinearity.
\end{abstract}


\maketitle

\section{Introduction}

In many realistic physical experiments involving wave propagation in thin waveguides,
multi-dimensional models can be reduced approximately to the one-dimensional PDEs on graphs
\cite{Beck,Joly1,Joly2,Kuch}. Similarly, quantum nanowires and other thin structures
in nanotechnology can be described by the one-dimensional Schr\"{o}dinger equation on graphs \cite{GS}.

Spectral properties of Laplacian and other linear operators on graphs have been intensively studied
in the past twenty year \cite{Kuchment,Exner}.
The time evolution of linear PDEs on graphs is well defined by the standard semi-group theory,
once a self-adjoint extension of the graph Laplacian is constructed. On the other hand,
the time evolution of nonlinear PDEs on graphs is a more challenging problem involving interplay
between nonlinear analysis, geometry, and the spectral theory of non-self-adjoint operators.
The nonlinear PDEs on graphs, mostly the nonlinear Schr\"{o}dinger equation (NLS), has been studied
in the past five years in the context of existence, stability, and propagation of solitary waves \cite{Noja}.

In a series of papers \cite{AdamiEPL,AdamiJPA,AdamiAH,AdamiJDE1,AdamiJDE2}, Adami, Cacciapuoti, Finco, and Noja
analyzed variational properties of stationary states on a star graph, which is the union of $N$ half-lines (edges) connected
at a single vertex. For the standard Kirchhoff boundary conditions at the vertex and for odd $N$,
there is only one stationary state of the NLS on the star graph. This state is represented by the half-solitons
along each edge glued by their unique maxima at the vertex. By using a one-parameter deformation
of the NLS energy constrained by the fixed mass, it was shown that
the half-soliton state is a saddle point of the constrained NLS energy \cite{AdamiJPA}. On the other hand,
by adding a focusing delta impurity to the vertex, it was proven that there exists a global minimizer
of the constrained NLS energy for a sufficiently small mass below the critical mass \cite{AdamiEPL,AdamiAH,AdamiJDE1}.
This minimizer coincides with the $N$-tail state symmetric under exchange of edges,
which has monotonically decaying tails and which becomes the half-soliton state if the
delta impurity vanishes. In the concluding paper \cite{AdamiJDE2}, it was proven that although
the constrained minimization problem admits no global minimizers for a sufficiently large mass
above the critical mass, the $N$-tail state symmetric under exchange of edges
is still a local minimizer of the constrained NLS energy when a focusing delta impurity is added to the vertex.

Due to local minimization property, the $N$-tail state symmetric under exchange of edges is
orbitally stable in the time evolution of the NLS in the presence of the focusing delta impurity. Although
the second variation of the constrained energy was mentioned in the first work \cite{AdamiEPL},
the authors obtained all the variational results in \cite{AdamiAH,AdamiJDE1,AdamiJDE2} from
the energy formulation avoiding the linearization procedure. In the same way, the
saddle point geometry of energy at the half-soliton state in the case of vanishing delta impurity
was not related in \cite{AdamiJPA} to the instability of the half-soliton state in the time evolution of the NLS.
It is quite well known that the saddle point geometry does not necessarily imply instability of
stationary states in Hamiltonian systems because of the presence of neutrally stable modes of negative
energy \cite{Kapitula}.

The recent works of Adami, Serra, and Tilli \cite{AdamiCV,AdamiJFA} were devoted to
the existence of ground states on the unbounded graphs that are connected to infinity after removal
of any edge. It was proven that the infimum of the constrained NLS energy on the unbounded graph coincides with the
infimum of the constrained NLS energy on the infinite line and it is not achieved (that is, no ground state exists)
for every such a graph with the exception of graphs isometric to the real line \cite{AdamiCV}.
The reason why the infimum is not achieved is a possibility to minimize the constrained NLS energy by a family
of NLS solitary waves escaping to infinity along one edge of the graph. The star graph
with vanishing delta impurity is an example of the unbounded graphs with no ground states, moreover,
the constrained NLS energy of the half-soliton state is strictly greater than its infimum.
Thus, the study in \cite{AdamiCV} provides a general argument of the computations in \cite{AdamiJPA}, where
it is shown that the one-parameter deformation of the half-soliton state with the fixed mass
reduces the NLS energy and connects the half-soliton state with the solitary wave escaping along one edge of the star graph.

Further works on existence and stability of stationary states on the unbounded graphs have been
developed in the context of the logarithmic NLS equation \cite{Ardila}, the power NLS equation with $\delta'$ interactions
\cite{Pava}, the power NLS equation on the tadpole graph \cite{NPS}, and the cubic NLS equation on the
periodic graph \cite{GilgPS,PS}.

In the present work, we provide a dynamical characterization of the result in \cite{AdamiJPA}
for the NLS with the power nonlinearity and in the case of an arbitrary star graph.
By using dynamical system methods (in particular, normal forms),
we will verify that the half-soliton state is the saddle point of the constrained NLS energy
on the star graph and moreover it is dynamically unstable due to the slow growth of perturbations.
This nonlinear instability is likely to result in the destruction of the half-soliton state pinned to the vertex
and the formation of a solitary wave escaping to infinity along one edge of the star graph.

Since the nonlinear saddle points are rarely met in applications of the NLS equations,
it is the first time to the best of our knowledge when the energy method is adopted
to the proof of the nonlinear instability of the stationary states.

The paper is organized as follows. Section 2 states the main results for the NLS equation on the star graph.
Positivity of the second variation of the action functional is proven in Section 3.
Saddle point geometry near the half-soliton state is proven with normal forms
in Section 4. Dynamical characterization of the nonlinear instability of the half-soliton state is developed
with normal forms and energy estimates in Section 5.

\section{Main results}

Let $\Gamma$ be a star graph, which is constructed by attaching $N$ half-lines
at a common vertex. Let us choose the vertex as the origin and parameterize each edge of $\Gamma$
by $\mathbb{R}^+$. The Hilbert space on the graph $\Gamma$ is given by
$$
L^2(\Gamma) = \oplus_{j=1}^N L^2(\mathbb{R}^+).
$$
Elements in $L^2(\Gamma)$ are represented in the componentwise sense as vectors $\Psi = (\psi_1, \psi_2, \dots, \psi_N)^T$
of $L^2(\mathbb{R}^+)$-functions with each component corresponding to one edge. The squared norm of
such $L^2(\Gamma)$-functions is given by
$$
\| \Psi \|_{L^2(\Gamma)}^2 := \sum_{j=1}^N \| \psi_j \|_{L^2(\mathbb{R}^+)}^2.
$$
Similarly, we define the $L^2$-based Sobolev spaces on the graph $\Gamma$
$$
H^k(\Gamma) = \oplus_{j=1}^N H^k(\mathbb{R}^+), \quad k \in \mathbb{N}
$$
and equip them with suitable boundary conditions at the vertex.
For the weak formulation of the NLS on $\Gamma$, we define
$H^1_{\Gamma}$ by using the continuity boundary conditions as in
\begin{equation}
\label{H1}
H_\Gamma^1 := \{ \Psi \in H^1(\Gamma): \quad \psi_1(0) = \psi_2(0) = \dots = \psi_N(0)\},
\end{equation}
whereas for the strong formulation of the NLS on $\Gamma$, we define
$H^2_{\Gamma}$ by using the {\em Kirchhoff} boundary conditions as in
\begin{equation}
\label{H2}
H_\Gamma^2 := \left\{ \Psi \in H^2(\Gamma): \quad \psi_1(0) = \psi_2(0) = \dots = \psi_N(0), \quad \sum_{j=1}^N \psi_j'(0) = 0 \right\}.
\end{equation}
The dual space to $H^1_{\Gamma}$ is $H^{-1}_{\Gamma} := H^{-1}(\Gamma)$.

We consider the NLS equation on the star graph $\Gamma$ with the power nonlinearity in the normalized form,
\begin{equation} \label{eq1}
 i \frac{\partial \Psi}{\partial t} = - \Delta \Psi - (p+1) |\Psi|^{2p} \Psi,
\end{equation}
where $p>0$, $\Psi = \Psi(t,x)$, $\Delta \Psi = (\psi_1'',\psi_2'',\dots,\psi_N'')^T$ is the Laplacian operator
defined in the componentwise sense with primes denoting derivatives in $x$,
and the nonlinear term $|\Psi|^{2p} \Psi$ is interpreted as a symbol for
$(|\psi_1|^{2p}\psi_1, |\psi_2|^{2p}\psi_2, \dots, |\psi_N|^{2p} \psi_N)^T$.

In the weak formulation, stationary states of the NLS are defined as
critical points of the action functional $\Lambda_{\omega}(\Psi) := E(\Psi) + \omega Q(\Psi)$
in the energy space $H^1_{\Gamma}$, where $\omega \in \mathbb{R}$ is a free parameter, whereas
\begin{equation}
\label{energy}
E(\Psi) = \| \Psi' \|_{L^2(\Gamma)}^2 - \| \Psi \|_{L^{2p+2}(\Gamma)}^{2p+2},
\quad Q(\Psi) = \| \Psi \|_{L^2(\Gamma)}^2
\end{equation}
are the energy and mass functionals, respectively. The local well-posedness of the NLS evolution in $H^1_{\Gamma}$
follows by a standard contraction method. The energy $E(\Psi)$ and mass $Q(\Psi)$ are constants of motion in the
time evolution of the NLS flow in $H^1_{\Gamma}$. See Propositions 2.1 and 2.2 in \cite{AdamiJDE1}.

\begin{remark}
\label{rem-global}
The local solutions of the NLS in $H^1_{\Gamma}$ are extended globally in time by the energy conservation
and Gagliardo--Nirenberg inequality for the $L^2$-subcritical power nonlinearity with $p \in (0,2)$.
On the other hand, local solutions to the NLS
are known to blow up in a finite time in the $H^1(\Gamma)$ norm for $p = 2$ (critical nonlinearity)
and $p > 2$ (supercritical nonlinearity).
\end{remark}

In the strong formulation, stationary states of the NLS are given by the standing wave solutions of the form
$$
\Psi(t,x) = e^{i\omega t} \Phi_{\omega}(x), \quad \Phi_{\omega} \in H^2_{\Gamma},
$$
where $(\omega,\Phi_{\omega})$ are real-valued solutions of the stationary NLS equation,
\begin{equation}\label{eq2}
-\Delta \Phi_{\omega} - (p+1) |\Phi_{\omega}|^{2p} \Phi_{\omega} = - \omega \Phi_{\omega}, \quad \Phi_{\omega} \in H^2_{\Gamma}.
\end{equation}
The weak and strong formulations of the stationary states of the NLS on $\Gamma$ are equivalent to each other because
the Kirchhoff conditions in $H^2_{\Gamma}$ are {\em natural} boundary conditions
for critical points of $\Lambda_{\omega}$ in $H^1_{\Gamma}$.

No solution $\Phi_{\omega} \in H^2_{\Gamma}$ to the stationary NLS equation (\ref{eq2}) exists for $\omega \leq 0$,
because $\sigma(-\Delta) \geq 0$ in $L^2(\Gamma)$ and $\Phi_{\omega}(x), \Phi_{\omega}'(x) \to 0$ as $x \to \infty$
if $\Phi_{\omega} \in H^2_{\Gamma}$ by Sobolev's embedding theorems. Therefore, we consider $\omega > 0$
in the stationary NLS equation (\ref{eq2}). Since $\Gamma$ consists of edges with the parametrization on $\mathbb{R}^+$,
the scaling transformation
\begin{equation}
\label{scaling-transform}
\Phi_{\omega}(x) = \omega^{\frac{1}{2p}} \Phi(z), \quad z = \omega^{\frac{1}{2}} x
\end{equation}
can be used to scale the positive parameter $\omega$ to unity. The normalized profile $\Phi$ is now a solution
of the stationary NLS equation
\begin{equation}\label{eq3}
-\Delta \Phi + \Phi - (p+1) |\Phi|^{2p} \Phi = 0, \quad \Phi \in H^2_{\Gamma}.
\end{equation}
The stationary NLS equation (\ref{eq3}) has a particular solution
\begin{equation}
\label{half-soliton}
\Phi(x) = \phi(x) \left( \begin{array}{c} 1 \\ 1 \\ \vdots \\ 1 \end{array} \right) \quad \mbox{\rm with} \;\;
\phi(x) = \sech^{\frac{1}{p}}(px).
\end{equation}
This solution is labeled as the {\em half-soliton state}.
If $N$ is odd, the half-soliton state is a unique solution to the stationary NLS equation (\ref{eq3}) in $H^2_{\Gamma}$,
whereas if $N$ is even, there exists additional solutions with the translational parameters.
See Theorem 5 in \cite{AdamiJDE1}. In what follows, we only consider the half-soliton state for any $N \geq 3$.

Our main results are given as follows. Thanks to the scaling transformation, we set $\omega = 1$
and use notations $\Lambda$ and $\Phi$ for $\Lambda_{\omega=1}$ and $\Phi_{\omega = 1}$.

\begin{theorem}
\label{theorem-positivity}
Let $\Lambda''(\Phi)$ be the Hessian operator for the second variation of
$\Lambda(\Psi)$ at $\Psi = \Phi$ in $H^1_{\Gamma}$.
For every $p \in (0,2)$, it is true that
$\langle \Lambda''(\Phi) V, V \rangle_{L^2(\Gamma)} \geq 0$
for every $V \in H^1_{\Gamma} \cap L^2_c$, where
\begin{equation}
\label{constraint}
L^2_c := \left\{ V \in L^2(\Gamma) : \quad \langle V, \Phi \rangle_{L^2(\Gamma)} = 0 \right\}.
\end{equation}
Moreover, $\langle \Lambda''(\Phi) V, V \rangle_{L^2(\Gamma)} = 0$
if and only if $V \in H^1_{\Gamma} \cap L^2_c$ belongs to
a $(N-1)$-dimensional subspace of $L^2_c$ spanned by an orthogonal
set $\{ U^{(1)}, U^{(2)}, \dots, U^{(N-1)} \}$.
Consequently, $V = 0$ is a degenerate minimizer of $\langle \Lambda''(\Phi) V, V \rangle_{L^2(\Gamma)}$
in $H^1_{\Gamma} \cap L^2_c$.
\end{theorem}

\begin{remark}
If $p = 2$, then $\langle \Lambda''(\Phi) V, V \rangle_{L^2(\Gamma)} = 0$ if and only if
$V \in H^1_{\Gamma} \cap L^2_c$ belongs to a $N$-dimensional subspace of $L^2_c$ with an additional
degeneracy. For $p > 2$, the second variation is not positive in $H^1_{\Gamma} \cap L^2_c$.
\end{remark}

\begin{theorem}
\label{theorem-saddle}
Let $X_c = {\rm span}\{U^{(1)}, U^{(2)}, \dots, U^{(N-1)}\} \subset L^2_c$
be defined in Theorem \ref{theorem-positivity}. For every $p \in \left[\frac{1}{2},2\right)$, 
there exists $\delta > 0$ such that for every
$c = (c_1,c_2,\dots,c_{N-1})^T \in \mathbb{R}^{N-1}$ satisfying
$\| c \| \leq \delta$, there exists a unique minimizer of the variational problem
\begin{equation}
\label{constrained-min}
M(c) := \inf_{U^{\perp} \in H^1_{\Gamma} \cap L^2_c \cap [X_c]^{\perp}} 
\left[ \Lambda(\Phi + c_1 U^{(1)} + c_2 U^{(2)} + \dots + c_{N-1} U^{(N-1)} + U^{\perp}) - \Lambda(\Phi) \right]
\end{equation}
such that $\| U^{\perp} \|_{H^1(\Gamma)} \leq A \| c \|^2$ for a $c$-independent constant $A > 0$.
Moreover, $M(c)$ is sign-indefinite in $c$.
Consequently, $\Phi$ is a nonlinear saddle point of $\Lambda$ in $H^1_{\Gamma}$
with respect to perturbations in $H^1_{\Gamma} \cap L^2_c$.
\end{theorem}

\begin{remark}
The restriction $p \geq \frac{1}{2}$ is used in order to expand $\Lambda(\Phi + U)$ up to the cubic terms
with respect to the perturbation $U \in H^1_{\Gamma} \cap L^2_c$ and then to pass to normal forms.
If $p = 2$, $\Phi$ is still a nonlinear saddle point of $\Lambda$ in $H^1_{\Gamma} \cap L^2_c$
but the proof needs to be modified by the fact that $X_c$ is $N$-dimensional.
If $p > 2$,  it follows already from the second derivative test that
$\Phi$ is a saddle point of $\Lambda$ in $H^1_{\Gamma} \cap L^2_c$.
\end{remark}

\begin{theorem}
\label{theorem-instability}
For every $p \in \left[\frac{1}{2},2\right)$, there exists $\epsilon > 0$
such that for every $\delta > 0$ (sufficiently small)
there exists $V \in H^1_{\Gamma}$ with $\| V \|_{H^1_{\Gamma}} \leq \delta$
such that the unique global solution $\Psi(t) \in C(\mathbb{R},H^1_{\Gamma}) \cap C^1(\mathbb{R},H^{-1}_{\Gamma})$
to the NLS equation (\ref{eq1}) starting with
the initial datum $\Psi(0) = \Phi + V$ satisfies
\begin{equation}
\label{orbital-instab}
\inf_{\theta \in \mathbb{R}} \| e^{-i \theta} \Psi(t_0) - \Phi \|_{H^1(\Gamma)} > \epsilon \quad \mbox{\rm for some \;} t_0 > 0.
\end{equation}
Consequently, the orbit $\{ \Phi e^{i \theta}\}_{\theta \in \mathbb{R}}$
is unstable in the time evolution of the NLS equation (\ref{eq1}) in $H^1_{\Gamma}$.
\end{theorem}

\begin{remark}
If $p = 2$, the instability claim of Theorem \ref{theorem-instability} follows
from analysis of the NLS equation on the real line \cite{Comech,Ohta}.
If $p > 2$, the instability claim of Theorem \ref{theorem-instability} follows
from the standard approach \cite{GSS}.
\end{remark}

Theorems \ref{theorem-positivity}, \ref{theorem-saddle}, \ref{theorem-instability} are proven in Sections
3,4, and 5, respectively.

\section{Proof of Theorem \ref{theorem-positivity}}

We set $\omega = 1$ by the scaling transformation (\ref{scaling-transform}) and consider the half-soliton state $\Phi$
given by (\ref{half-soliton}). Substituting $\Psi = \Phi + U + iW$ with real-valued $U,W \in H^1_{\Gamma}$ into
the action functional $\Lambda(\Psi) = E(\Psi) + Q(\Psi)$ and expanding in $U,W$ yield
\begin{equation}
\label{second-variation-Lambda}
\Lambda(\Phi + U + i W) = \Lambda(\Phi) + \langle L_+ U, U \rangle_{L^2(\Gamma)} +
\langle L_- W, W \rangle_{L^2(\Gamma)} + N(U,W),
\end{equation}
where
\begin{eqnarray*}
\langle L_+ U, U \rangle_{L^2(\Gamma)} & := & \int_{\Gamma} \left[ (\nabla U)^2 + U^2 - (2p+1) (p+1) \Phi^{2p} U^2 \right] dx, \\
\langle L_- W, W \rangle_{L^2(\Gamma)} & := & \int_{\Gamma} \left[ (\nabla W)^2 + W^2 - (p+1) \Phi^{2p} W^2 \right] dx,
\end{eqnarray*}
and
\begin{eqnarray*}
N(U,W) = \left\{ \begin{array}{l} {\rm o}(\| U + i W \|_{H^1(\Gamma)}^2), \quad p \in \left(0,\frac{1}{2}\right), \\
{\rm O}(\| U + i W \|_{H^1(\Gamma)}^3), \quad p \geq \frac{1}{2}. \end{array} \right.
\end{eqnarray*}
In the strong formulation, we can also define the two Hessian operators
\begin{eqnarray}
\label{Lplus}
L_+ & = & -\Delta + 1 - (2p+1) (p+1) \Phi^{2p} : H^2_{\Gamma} \subset L^2(\Gamma) \to L^2(\Gamma), \\
L_- & = & -\Delta + 1 - (p+1) \Phi^{2p} : \phantom{texttext} H^2_{\Gamma} \subset L^2(\Gamma) \to L^2(\Gamma),
\label{Lminus}
\end{eqnarray}
where $\Phi^{2p} = (\phi_1^{2p}, \phi_2^{2p}, \dots, \phi_N^{2p})^T$. Both operators are extended as the self-adjoint
operators in $L^2(\Gamma)$. The spectrum $\sigma(L_{\pm}) \subset \mathbb{R}$ consists of the continuous and discrete parts
denoted by $\sigma_c(L_{\pm})$ and $\sigma_p(L_{\pm})$ respectively.

By Weyl's Theorem, since $\Phi^{2p}$ is bounded and decays exponentially fast to zero
at infinity, we have $\sigma_c(L_{\pm}) = \sigma(- \Delta + 1) = [1,\infty)$.
Therefore, we are only concerned with the eigenvalues of $\sigma_p(L_{\pm})$ in $(-\infty,1)$.
The following result shows that $\sigma_p(L_-) \geq 0$, $0 \in \sigma_p(L_-)$ is simple,
and $L_-$ is coercive in the subspace $L^2_c$ associated with a single constraint in (\ref{constraint}).

\begin{lemma}
\label{prop-coercivity-L-minus}
There exists $C > 0$ such that
\begin{equation}
\label{coercivity}
\langle L_- W, W \rangle_{L^2(\Gamma)} \geq C \| W \|_{H^1(\Gamma)}^2 \quad \mbox{\rm for every \;} W \in H^1_{\Gamma} \cap L^2_c,
\end{equation}
where $L^2_c$ is given by (\ref{constraint}).
\end{lemma}

\begin{proof}
By using (\ref{half-soliton}), we write for every $W = (w_1, w_2, \dots, w_N)^T \in H^1_{\Gamma}$,
\begin{equation} \label{operLminus}
\langle L_- W, W \rangle_{L^2(\Gamma)} = \sum_{j=1}^{N} \int_0^{+\infty} \Big[ \left(\frac{d w_j}{dx} \right)^2
+ w_j^2 - (p+1) \phi^{2p} w_j^2 \Big] dx.
\end{equation}
By using $\phi'' = \phi - (p+1) \phi^{2p+1}$, $(\phi')^2 = \phi^2 - \phi^{2p+2}$, and
integration by parts, we obtain
$$
\int_0^{+\infty} p w_j^2 \phi^{2p} dx = \int_0^{+\infty} 2 w_j \frac{d w_j}{dx} \frac{\phi'}{\phi} dx
$$
and
$$
\int_0^{+\infty} \left( w_j^2 - \phi^{2p} w_j^2 \right) dx = \int_0^{+\infty} w_j^2 \left( \frac{\phi'}{\phi} \right)^2 dx,
$$
so that the representation (\ref{operLminus}) is formally equivalent to
\begin{equation}
\label{representation-L-minus}
\langle L_- W, W \rangle_{L^2(\Gamma)} = \sum_{j=1}^{N} \int_0^{+\infty} \phi^2 \Big| \frac{d}{dx} \Big( \frac{w_j}{\phi} \Big) \Big|^2 dx \geq 0.
\end{equation}
Since $\phi(x) > 0$ for every $x \in \mathbb{R}^+$ and $\partial_x \log\phi \in L^{\infty}(\mathbb{R})$,
the representation (\ref{representation-L-minus}) is justified for every $W \in H^1_{\Gamma}$.
It follows from (\ref{representation-L-minus}) that $\langle L_- W, W \rangle_{L^2(\Gamma)} = 0$
if and only if $W \in H^1_{\Gamma}$ satisfies
\begin{equation}
\label{representation-L-zero}
 \frac{d}{dx} \Big( \frac{w_j}{\phi} \Big) = 0 \quad \mbox{\rm almost everywhere and for every \;} j.
\end{equation}
Sobolev's embedding of $H^1(\mathbb{R}^+)$ into $C(\mathbb{R}^+)$ and
equation (\ref{representation-L-zero}) imply that $w_j = c_j \phi$ for some constant $c_j$.
The continuity boundary conditions in the definition of $H^1_{\Gamma}$ in (\ref{H1}) then
yield $c_1 = c_2 = \dots = c_N$, which means that $0$ is a simple eigenvalue of the operator $L_-$ in (\ref{Lminus})
with the eigenvector $\Phi$. Since eigenvalues of $\sigma_p(L_-) \in (-\infty,1)$ are isolated,
the coercivity bound (\ref{coercivity}) follows by the
spectral theorem and G{\aa}rding inequality.
\end{proof}

In order to study $\sigma_p(L_+)$ in $(-\infty,1)$, we recall the construction of exponentially decaying solutions
of the second-order differential equation
\begin{equation} \label{eqofcomp}
- u''(x) + u(x) - (2p+1)(p+1) \sech^2(px) u(x) = \lambda u(x), \quad x \in (0,\infty), \quad \lambda < 1.
\end{equation}
The following two lemmas recall some well-known results on the Schr\"{o}dinger equation (\ref{eqofcomp}).

\begin{lemma}\label{solhyp}
For every $\lambda<1$, there exists a unique solution $u \in C^1(\mathbb{R}^+)$ to equation (\ref{eqofcomp})
such that
\begin{equation}
\label{u-limit}
\lim_{x \to +\infty} u(x) e^{\sqrt{1-\lambda} x} = 1.
\end{equation}
The other linearly independent solution to equation (\ref{eqofcomp}) diverges as $x \to +\infty$.
\end{lemma}

\begin{proof}
See, e.g., Lemma 5.2 in \cite{NPS}. The existence of a unique decaying solution as $x \to +\infty$
is obtained after the boundary--value problem (\ref{eqofcomp})--(\ref{u-limit}) is reformulated
as the Volterra's integral equation with a bounded kernel. The other linearly independent solution
to the second-order equation (\ref{eqofcomp}) diverges as $x \to +\infty$ thanks
to the $x$-independent and nonzero Wronskian determinant between the two solutions.
\end{proof}

\begin{lemma}
\label{lem-Sturm}
If $u(0) = 0$ (resp. $u'(0) = 0$) for some $\lambda_0 < 1$, the solution $u$ of Lemma \ref{solhyp} is extended into
an odd (resp. even) eigenfunction of the Schr\"{o}dinger equation (\ref{eqofcomp}) on the infinite line.
The point $\lambda_0$ becomes the eigenvalue of the associated Schr\"{o}dinger operator defined in $L^2(\mathbb{R})$.
There exists exactly one $\lambda_0 < 0$ corresponding to $u'(0) = 0$ and a simple eigenvalue $\lambda_0 = 0$
corresponding to $u(0) = 0$, all other such points $\lambda_0$ in $(0,1)$ are bounded away from zero.
\end{lemma}

\begin{proof}
The solution $u$ in Lemma \ref{solhyp} is extended
to an eigenfunction of the associated Schr\"{o}dinger operator defined in $L^2(\mathbb{R})$
by the reversibility of the Schr\"{o}dinger equation (\ref{eqofcomp}) with
respect to the transformation $x \mapsto -x$. The count of eigenvalues follows by Sturm's nodal theorem since
the odd eigenfunction for the eigenvalue $\lambda_0 = 0$,
$$
\phi'(x) = - \sech^{\frac{1}{p}}(px) \tanh(px)
$$
has one zero on the infinite line. Hence, $\lambda_0 = 0$ is the second eigenvalue of
the associated Schr\"{o}dinger equation with exactly one simple negative eigenvalue $\lambda_0 < 0$
that corresponds to an even eigenfunction. All other eigenvalues in $(0,1)$ are bounded away from zero.
\end{proof}

\begin{remark}
For $p = 1$, the solution $u$ in Lemma \ref{solhyp} is available in the closed analytic form:
$$
u(x) = e^{-\sqrt{1-\lambda}x} \frac{3-\lambda + 3 \sqrt{1-\lambda} \tanh x - 3 \sech^2 x}{3 - \lambda + 3 \sqrt{1-\lambda}}.
$$
In this case, the eigenvalues and eigenfunctions in Lemma \ref{lem-Sturm} are given by
\begin{eqnarray*}
\lambda = -3 : & \quad u(x) = \frac{1}{4} \sech^2 x, \\
\lambda = 0 : & \quad \qquad u(x) = \frac{1}{2} \tanh x \sech x.
\end{eqnarray*}
No other eigenvalues of the associated Schr\"{o}dinger operator on $L^2(\mathbb{R})$ exist in $(-\infty,1)$.
\end{remark}

By using Lemmas \ref{solhyp} and \ref{lem-Sturm}, we can now characterize
$\sigma_p(L_+)$ in $(-\infty,1)$. The following result shows that $\sigma_p(L_+)$
includes a simple negative eigenvalue and a zero eigenvalue
of multiplicity $N-1$.

\begin{lemma}
\label{lem-eig}
Let $u$ be a solution of Lemma \ref{solhyp} for $\lambda \in (-\infty,1)$.
Then, $\lambda_0 \in (-\infty,1)$ is an eigenvalue of $\sigma_p(L_+)$
if and only if either $u(0) = 0$ or $u'(0) = 0$ (both $u(0)$ and $u'(0)$ cannot vanish simultaneously).
Moreover, $\lambda_0$ in $\sigma_p(L_+)$ has
multiplicity $N-1$ if $u(0) = 0$ and multiplicity $1$ if $u'(0) = 0$.
\end{lemma}

\begin{proof}
Let $\lambda_0 \in (-\infty,1)$ be an eigenvalue of $\sigma_p(L_+)$ and denote
the eigenvector by $U \in H^2_{\Gamma}$. Since $U(x)$ and $U'(x)$ decay to zero as $x \to +\infty$,
by Sobolev's embedding of $H^2(\mathbb{R}^+)$ to the space $C^1(\mathbb{R}^+)$,
we can parameterize $U \in H^2_{\Gamma}$
by using $u$ from Lemma \ref{solhyp} as follows
$$
U(x) = u(x) \begin{pmatrix} c_1 \\ c_2 \\ \vdots \\ c_N \end{pmatrix},
$$
where $(c_1,c_2,\dots,c_N)$ are some coefficients. By using the boundary conditions
in the definition of $H^2_{\Gamma}$ in (\ref{H2}), we obtain a homogeneous linear system
on the coefficients:
\begin{equation}
\label{lin-system}
c_1 u(0) = c_2 u(0) = \dots = c_N u(0), \quad c_1 u'(0) + c_2 u'(0) + \dots + c_N u'(0) = 0.
\end{equation}
The determinant of the associated matrix is
\begin{equation}
\label{determinant}
\Delta = [u(0)]^{N-1} u'(0) \left| \begin{array}{ccccc} 1 & -1 & 0 & \dots & 0 \\
1 & 0 & -1 & \dots & 0 \\ 1 & 0 & 0 & \dots & 0 \\ \vdots & \vdots & \vdots & \ddots & \vdots \\
1 & 1 & 1 & \dots & 1
\end{array} \right| = N [u(0)]^{N-1} u'(0).
\end{equation}
Therefore, $U \neq 0$ is an eigenvector for an eigenvalue $\lambda_0 \in (-\infty,1)$
if and only if $\Delta = 0$, which is only possible in (\ref{determinant}) if either $u(0) = 0$ or $u'(0) = 0$.
Moreover, multiplicity of $u(0)$ and $u'(0)$ in $\Delta$ coincides with the
multiplicity of the eigenvalue $\lambda_0$ because it gives the number of
linearly independent solutions of the homogeneous linear system (\ref{lin-system}).
The assertion of the lemma is proven.
\end{proof}

\begin{corollary}
\label{cor-eig}
There exists exactly one simple negative eigenvalue $\lambda_0 < 0$ in $\sigma_p(L_+)$ and
a zero eigenvalue $\lambda_0 = 0$ in $\sigma_p(L_+)$
of multiplicity $N-1$, all other possible eigenvalues of $\sigma_p(L_+)$ in $(0,1)$ are bounded away from zero.
\end{corollary}

\begin{proof}
The result follows from Lemmas \ref{lem-Sturm} and \ref{lem-eig}.
\end{proof}

\begin{remark}
For the simple eigenvalue $\lambda_0 < 0$ in $\sigma_p(L_+)$, the corresponding eigenvector is
$$
U = u(x) \begin{pmatrix} 1 \\ 1 \\ \vdots \\ 1 \end{pmatrix},
$$
where $u(x) > 0$ for every $x \in \mathbb{R}^+$ with $u'(0) = 0$. For the eigenvalue
$\lambda_0 = 0$ of multiplicity $N-1$ in $\sigma_p(L_+)$, the invariant subspace of $L_+$
can be spanned by an orthogonal basis of eigenvectors $\{ U^{(1)}, U^{(2)}, \dots, U^{(N-1)}\}$.
The orthogonal basis of eigenvectors can be constructed by induction as follows:
\begin{eqnarray*}
& N = 3 : & \quad U^{(1)} = \phi'(x) \begin{pmatrix} 1 \\ -1 \\ 0 \end{pmatrix}, \quad
U^{(2)} = \phi'(x) \begin{pmatrix} 1 \\ 1 \\ -2 \end{pmatrix},\\
& N = 4 : & \quad U^{(1)} = \phi'(x) \begin{pmatrix} 1 \\ -1 \\ 0 \\ 0 \end{pmatrix}, \quad
U^{(2)} = \phi'(x) \begin{pmatrix} 1 \\ 1 \\ -2 \\ 0 \end{pmatrix}, \quad
U^{(3)} = \phi'(x) \begin{pmatrix} 1 \\ 1 \\ 1 \\ -3 \end{pmatrix},
\end{eqnarray*}
and so on. \label{rem-vectors}
\end{remark}

The following result shows that the operator $L_+$ is positive in the subspace $L^2_c$
associated with a scalar constraint in (\ref{constraint}), provided
the nonlinearity power $p$ is in $(0,2)$, and coercive on a subspace of $L^2_c$ orthogonal to ${\rm ker}(L_+)$.

\begin{lemma}
\label{prop-coercivity-L-plus}
For every $p \in (0,2)$, $\langle L_+ U, U \rangle_{L^2(\Gamma)} \geq 0$ for every $U \in H^1_{\Gamma} \cap L^2_c$,
where $L^2_c$ is given by (\ref{constraint}). Moreover $\langle L_+ U, U \rangle_{L^2(\Gamma)} = 0$ if and only if
$U \in H^1_{\Gamma} \cap L^2_c$ belongs to the $(N-1)$-dimensional subspace
$X_c = {\rm span}\{U^{(1)}, U^{(2)}, \dots, U^{(N-1)}\} \subset L^2_c$ in the kernel of $L_+$.
Consequently, there exists $C_p > 0$ such that
\begin{equation}
\label{coercivity-L-plus}
\langle L_+ U, U \rangle_{L^2(\Gamma)} \geq C_p \| U \|_{H^1(\Gamma)}^2 \quad \mbox{\rm for every \;} U \in H^1_{\Gamma} \cap L^2_c \cap [X_c]^{\perp}.
\end{equation}
\end{lemma}

\begin{proof}
Since $\sigma_c(L_+) = \sigma(- \Delta + 1) = [1,\infty)$ by Weyl's Theorem, the eigenvalues
of $\sigma_p(L_+)$ at $\lambda_0 < 0$ and $\lambda = 0$ given by Corollary \ref{cor-eig} are isolated.
Since $\langle U^{(k)}, \Phi \rangle_{L^2(\Gamma)} = 0$ for every $1 \leq k \leq N-1$, $L_+^{-1} \Phi$ exists
in $L^2(\Gamma)$ and is in fact given by $L_+^{-1} \Phi = -\partial_{\omega} \Phi_{\omega} |_{\omega = 1}$
up to an addition of an arbitrary element in ${\rm ker}(L_+)$.
By the well-known result (see Theorem 3.3 in \cite{GSS}), $L_+ |_{L^2_c}$ (that is,
$L_+$ restricted on subspace $L^2_c$) is nonnegative if and only if
\begin{equation}
\label{positivity-Phi}
0 \geq \langle L_+^{-1} \Phi, \Phi \rangle_{L^2(\Gamma)} = - \langle \partial_{\omega} \Phi_{\omega} |_{\omega = 1}, \Phi \rangle_{L^2(\Gamma)}
= -\frac{1}{2} \frac{d}{d \omega} \| \Phi_{\omega} \|_{L^2(\Gamma)}^2 \biggr|_{\omega = 1}.
\end{equation}
Moreover, ${\rm ker}(L_+ |_{L^2_c}) = {\rm ker}(L_+)$ if $\langle L_+^{-1} \Phi, \Phi \rangle_{L^2(\Gamma)} \neq 0$.
By the direct computation, we obtain
$$
\| \Phi_{\omega} \|_{L^2(\Gamma)}^2 = N \omega^{\frac{1}{p}-\frac{1}{2}} \int_0^{\infty} \phi(z)^2 dz
$$
so that
\begin{equation}
\label{norm-squared}
 \frac{d}{d \omega} \| \Phi_{\omega} \|_{L^2(\Gamma)}^2 = N \left( \frac{1}{p} - \frac{1}{2} \right)
\omega^{\frac{1}{p}-\frac{3}{2}} \int_0^{\infty} \phi(z)^2 dz,
\end{equation}
so that $L_+ |_{L^2_c} \geq 0$ if $p \in (0,2]$ and ${\rm ker}(L_+ |_{L^2_c}) = {\rm ker}(L_+)$ if $p \in (0,2)$.
This argument gives the first two assertions of the lemma. The coercivity bound (\ref{coercivity-L-plus})
follows from the spectral theorem in $L^2_c$ and G{\aa}rding inequality.
\end{proof}

\begin{proof1}{\em of Theorem \ref{theorem-positivity}.}
It follows from the expansion (\ref{second-variation-Lambda}) that
$$
\frac{1}{2} \langle \Lambda''(\Phi) V, V \rangle_{L^2(\Gamma)} =
\langle L_+ U, U \rangle_{L^2(\Gamma)} +
\langle L_- W, W \rangle_{L^2(\Gamma)} \quad \mbox{\rm with} \;\; V = U + i W,
$$
where $U,W \in H^1_{\Gamma}$ are real-valued. The result of Theorem \ref{theorem-positivity}
follows by Lemmas \ref{prop-coercivity-L-minus} and \ref{prop-coercivity-L-plus}.
\end{proof1}

\section{Proof of Theorem \ref{theorem-saddle}}

To prove Theorem \ref{theorem-saddle}, it is sufficient to work with real-valued perturbations $U \in H^1_{\Gamma} \cap L^2_c$
to the critical point $\Phi \in H^1_{\Gamma}$ of the action functional $\Lambda$.
Assuming $p \geq \frac{1}{2}$, we substitute $\Psi = \Phi + U$ with real-valued $U \in H^1_{\Gamma}$
into $\Lambda(\Psi)$ and expand in $U$ to obtain
\begin{equation}
\label{Lambda-nonlinear}
\Lambda(\Phi + U) = \Lambda(\Phi) + \langle L_+ U, U \rangle_{L^2(\Gamma)} - \frac{2}{3} p(p+1)(2p+1)\langle \Phi^{2p-1} U^2, U \rangle_{L^2(\Gamma)}
+ S(U),
\end{equation}
where
\begin{eqnarray*}
S(U) = \left\{ \begin{array}{l} {\rm o}(\| U \|_{H^1(\Gamma)}^3), \quad p \in \left(\frac{1}{2},1\right), \\
{\rm O}(\| U \|_{H^1(\Gamma)}^4), \quad p \geq 1. \end{array} \right.
\end{eqnarray*}
Compared to the expansion (\ref{second-variation-Lambda}), we have set $W = 0$ and
have expanded the cubic term explicitly, under the additional assumption $p \geq \frac{1}{2}$.
In what follows, we inspect convexity of $\Lambda(\Phi + U)$ with respect
to the small perturbation $U \in H^1_{\Gamma} \cap L^2_c$.

The quadratic form $\langle L_+ U, U \rangle_{L^2(\Gamma)}$ is associated with the same operator $L_+$ given by (\ref{Lplus}).
By Lemma \ref{prop-coercivity-L-plus}, ${\rm ker}(L_+) \equiv X_c = {\rm span}\{U^{(1)},U^{(2)},\dots,U^{(N-1)}\}$
for every $p > 0$, where the orthogonal vectors $\{U^{(1)},U^{(2)},\dots,U^{(N-1)}\}$
are constructed inductively in Remark \ref{rem-vectors}. Furthermore,
by Lemma \ref{prop-coercivity-L-plus}, if $U \in H^1_{\Gamma} \cap L^2_c$,
that is, if $U$ satisfies $\langle U, \Phi \rangle_{L^2(\Gamma)} = 0$, then
the quadratic form $\langle L_+ U, U \rangle_{L^2(\Gamma)}$ is positive for $p \in (0,2)$,
whereas if $U \in H^1_{\Gamma} \cap L^2_c \cap [X_c]^{\perp}$, the quadratic form is coercive.
Hence, we inspect positivity of $\Lambda(\Phi + U) - \Lambda(\Phi)$ and
use the orthogonal decomposition for $U \in H^1_{\Gamma} \cap L^2_c$:
\begin{equation}
\label{decomposition-U}
U = c_1 U^{(1)} + c_2 U^{(2)} + \dots + c_{N-1} U^{(N-1)} + U^{\perp},
\end{equation}
where $U^{\perp} \in H^1_{\Gamma} \cap L^2_c \cap [X_c]^{\perp}$. Therefore, $\langle U^{\perp},U^{(j)} \rangle_{L^2(\Gamma)} = 0$ for every $j$
and the coefficients $(c_1,c_2,\dots,c_{N-1})$ are found uniquely by
$$
c_j = \frac{\langle U, U^{(j)} \rangle_{L^2(\Gamma)}}{\| U^{(j)} \|^2_{L^2(\Gamma)}}, \quad \mbox{\rm for every \;} j.
$$
The following result shows how to define a unique mapping from $c = (c_1,c_2,\dots,c_{N-1})^t \in \mathbb{R}^{N-1}$
to $U^{\perp} \in H^1_{\Gamma} \cap L^2_c \cap [X_c]^{\perp}$ for small $c$.

\begin{lemma}
\label{prop-minimization}
For every $p \in \left[ \frac{1}{2},2\right)$, there exists $\delta > 0$ and $A > 0$
such that for every $c \in \mathbb{R}^{N-1}$ satisfying $\| c \| \leq \delta$,
there exists a unique minimizer $U^{\perp} \in H^1_{\Gamma} \cap L^2_c \cap [X_c]^{\perp}$
of the variational problem
\begin{equation}
\label{min-problem}
\inf_{U^{\perp} \in H^1_{\Gamma} \cap L^2_c \cap [X_c]^{\perp}}
\left[ \Lambda(\Phi + c_1 U^{(1)} + c_2 U^{(2)} + \dots + c_{N-1} U^{(N-1)} + U^{\perp}) - \Lambda(\Phi) \right].
\end{equation}
satisfying
\begin{equation}
\label{bound-on-U-perp}
\| U^{\perp} \|_{H^1(\Gamma)} \leq A \| c \|^2.
\end{equation}
\end{lemma}

\begin{proof}
First, we find the critical point of $\Lambda(\Phi + U)$ with respect to $U^{\perp} \in H^1_{\Gamma} \cap L^2_c \cap [X_c]^{\perp}$
for a given small $c \in \mathbb{R}^{N-1}$. Therefore, we set up the Euler--Lagrange equation in the form
$F(U^{\perp}, c) = 0$, where
\begin{equation}
\label{ELeq}
F(U^{\perp},c) : X \times \mathbb{R}^{N-1} \mapsto Y, \quad
X := H^1_{\Gamma} \cap L^2_c \cap [X_c]^{\perp}, \quad
Y := H^{-1}_{\Gamma} \cap L^2_c \cap [X_c]^{\perp}
\end{equation}
is given explicitly by
\begin{equation*}
F(U^\perp, c) := L_+ U^{\perp} - p(p+1)(2p+1) \Pi_c \Phi^{2p-1} \left( \sum_{j=1}^{N-1} c_j U^{(j)} + U^{\perp} \right)^2
- \Pi_c R\left( \sum_{j=1}^{N-1} c_j U^{(j)} + U^{\perp} \right),
\end{equation*}
where $\Pi_c : L^2(\Gamma) \mapsto L^2_c \cap [X_c]^{\perp}$ is the orthogonal projection operator
and $R(U)$ satisfies
\begin{eqnarray*}
\| R(U) \|_{H^1(\Gamma)} = \left\{ \begin{array}{l} {\rm o}(\| U \|_{H^1(\Gamma)}^2), \quad p \in \left(\frac{1}{2},1\right), \\
{\rm O}(\| U \|_{H^1(\Gamma)}^3), \quad p \geq 1. \end{array} \right.
\end{eqnarray*}
Operator function $F$ satisfies the conditions of the implicit function theorem:
\begin{itemize}
\item[(i)] $F$ is a $C^2$ map from $X \times \mathbb{R}^{N-1}$ to $Y$;
\item[(ii)] $F(0,0) = 0$;
\item[(iii)] $D_{U^\perp} F(0,0) = \Pi_c L_+ \Pi_c : X \mapsto Y$ is invertible with a bounded inverse from $Y$ to $X$.
\end{itemize}
By the implicit function theorem (see Theorem 4.E in \cite{Zeidel}), there are $r>0$ and $\delta>0$ such that
for each $c \in \mathbb{R}^{N-1}$ with $\| c \| \leq \delta$ there exists a unique solution
$U^\perp \in X$ of the operator equation $F(U^\perp, c) = 0$ with $\| U^\perp \|_{H^1(\Gamma)} \leq r$
such that the map
\begin{equation}
\label{map-U-perp}
\mathbb{R}^{N-1}  \ni c  \to U^\perp(c) \in X
\end{equation}
is $C^2$ near $c = 0$ and $U^{\perp}(0) = 0$. Since $D_{U^\perp} F(0,0) = \Pi_c L_+ \Pi_c : X \mapsto Y$
is strictly positive, the associated quadratic form is coercive according to the bound (\ref{coercivity-L-plus}),
hence the critical point $U^{\perp} = U^{\perp}(c)$ is a unique infimum of the variational problem (\ref{min-problem})
near $c = 0$.

It remains to prove the bound (\ref{bound-on-U-perp}). To show this, we note that
$$
F(0,c) = - p(p+1)(2p+1) \Pi_c \Phi^{2p-1} \left( \sum_{j=1}^{N-1} c_j U^{(j)} \right)^2
- \Pi_c R\left( \sum_{j=1}^{N-1} c_j U^{(j)}\right)
$$
satisfies $\| F(0,c) \|_{L(\Gamma)} \leq \tilde{A} \|c\|^2$ for a $c$-independent constant $\tilde{A} > 0$.
Since $F$ is a $C^2$ map from $X \times \mathbb{R}^{N-1}$ to $Y$ and $D_c F(0,0) = 0$,
we have $D_c U^{\perp}(0) = 0$, so that the $C^2$ map (\ref{map-U-perp}) satisfies the bound (\ref{bound-on-U-perp}).
\end{proof}

\begin{proof1}{\em of Theorem \ref{theorem-saddle}.}
Let us denote
\begin{equation}
\label{minimum-U-perp}
M(c) := \inf_{U^{\perp} \in H^1_{\Gamma} \cap L^2_c \cap [X_c]^{\perp}}
\left[ \Lambda(\Phi + c_1 U^{(1)} + c_2 U^{(2)} + \dots + c_{N-1} U^{(N-1)} + U^{\perp}) - \Lambda(\Phi) \right],
\end{equation}
where the infimum is achieved by Lemma \ref{prop-minimization} for sufficiently small $c \in \mathbb{R}^{N-1}$.
Thanks to the representation (\ref{Lambda-nonlinear}) and the bound (\ref{bound-on-U-perp}),
we obtain $M(c) = M_0(c) + \widetilde{M}(c)$, where
\begin{equation}
\label{energy-cube}
M_0(c) :=  - \frac{2}{3} p(p+1)(2p+1) \sum_{i=1}^{N-1} \sum_{j=1}^{N-1} \sum_{k=1}^{N-1} c_i c_j c_k
\langle \Phi^{2p-1} U^{(i)} U^{(j)}, U^{(k)} \rangle_{L^2(\Gamma)}
\end{equation}
and
\begin{eqnarray*}
\tilde{M}(c) = \left\{ \begin{array}{l} {\rm o}(\| c\|^3), \quad p \in \left(\frac{1}{2},1\right), \\
{\rm O}(\| c\|^4), \quad p \geq 1. \end{array} \right.
\end{eqnarray*}
In order to show that $M_0(c)$ is sign-indefinite near $c = 0$,
it is sufficient to show that at least one diagonal cubic coefficient in $M_0(c)$ is nonzero. Since
\begin{eqnarray*}
\int_0^{+\infty} \phi^{2p-1} (\phi')^3 dx = -\int_0^{+\infty} \sech^{\frac{2p+2}{p}}(px) \tanh^3(px) dx = -\frac{p}{2(p+1)(2p+1)},
\end{eqnarray*}
we obtain
\begin{equation}
\label{energy-cube-nonzero}
\langle \Phi^{2p-1} U^{(j)} U^{(j)}, U^{(j)} \rangle_{L^2(\Gamma)} = \frac{p j (j^2 - 1)}{2(p+1)(2p+1)} \neq 0, \quad j \geq 2,
\end{equation}
where the algorithmic construction of the orthogonal vectors $\{U^{(1)},U^{(2)},\dots,U^{(N-1)}\}$
in Remark \ref{rem-vectors} has been used. Since the diagonal coefficients in front of the cubic terms
$c_2^3, c_3^3, \dots, c_{N-1}^3$ in $M_0(c)$ are nonzero, $M_0(c)$ and hence $M(c)$ is sign-indefinite near $c = 0$.
\end{proof1}

\begin{remark}
We give explicit expressions for the function $M_0(c)$:
\begin{eqnarray*}
& N = 3 : & \quad M_0(c) = 2p^2 (c_1^2 - c_2^2) c_2, \\
& N = 4 : & \quad M_0(c) = 2p^2 ( c_1^2 c_2 + c_1^2 c_3 - c_2^3 + 3 c_2^2 c_3 - 4 c_3^3),
\end{eqnarray*}
and so on.
\end{remark}

\section{Proof of Theorem \ref{theorem-instability}}

We develop the proof of instability of the nonlinear saddle point $\Phi$ of the constrained action functional
$\Lambda$ by using the energy method. The steps in the proof of Theorem \ref{theorem-instability} are as follows.

First, we use the gauge symmetry
and project a unique global solution to the NLS equation (\ref{eq1}) with $p \in (0,2)$ in $H^1_{\Gamma}$
to the modulated stationary states $\{ e^{i \theta} \Phi_{\omega}\}_{\theta,\omega}$ with $\omega$ near $\omega_0 = 1$
and the symplectically orthogonal remainder term $V$.
Second, we project the remainder term $V$ into the $2(N-1)$-dimensional subspace associated
with the $(N-1)$-dimensional subspace $X_c$ defined in Theorem \ref{theorem-saddle}
and the symplectically orthogonal complement $V^{\perp}$.
Third, we define a truncated
Hamiltonian system of $(N-1)$ degrees of freedom for the coefficients of the projection on $X_c$.
Fourth, we use the energy conservation to control
globally the time evolution of $\omega$ and $V^{\perp}$ in terms of
the initial conditions and the reduced energy for the finite-dimensional Hamiltonian
system.
Finally, we transfer the instability of the zero equilibrium in
the finite-dimensional system to the
instability result (\ref{orbital-instab}) for the NLS equation (\ref{eq1}).

\subsection{Step 1: Modulated stationary states and a symplectically orthogonal remainder}

We start with the standard result, which holds if
$\langle \Phi_{\omega}, \partial_{\omega} \Phi_{\omega} \rangle_{L^2(\Gamma)} \neq 0$.

\begin{lemma}
\label{lem-orthogonal}
For every $p \in (0,2)$, there exists $\delta_0 > 0$ such that
for every $\Psi \in H^1_{\Gamma}$ satisfying
\begin{equation}
\label{orth-given}
\delta := \inf_{\theta \in \mathbb{R}} \| e^{-i \theta} \Psi - \Phi \|_{H^1(\Gamma)} \leq \delta_0,
\end{equation}
there exists a unique choice for real-valued $(\theta,\omega)$ and real-valued $U,W \in H^1_{\Gamma}$
in the orthogonal decomposition
\begin{equation}
\label{orth-decomposition}
\Psi = e^{i \theta} \left[ \Phi_{\omega} + U + i W \right], \quad \langle U, \Phi_{\omega} \rangle_{L^2(\Gamma)} =
\langle W, \partial_{\omega} \Phi_{\omega} \rangle_{L^2(\Gamma)} = 0,
\end{equation}
satisfying the estimate
\begin{equation}
\label{orth-bound}
| \omega - 1| + \| U + i W \|_{H^1(\Gamma)} \leq C \delta,
\end{equation}
for some positive constant $C > 0$.
\end{lemma}

\begin{proof}
Let us define the following vector function $G(\theta,\omega;\Psi) : \mathbb{R}^2 \times H^1_{\Gamma} \mapsto \mathbb{R}^2$ given by
$$
G(\theta,\omega;\Psi) := \left[ \begin{array}{c} \langle {\rm Re}(e^{-i \theta} \Psi - \Phi_{\omega}), \Phi_{\omega} \rangle_{L^2(\Gamma)} \\
 \langle {\rm Im}(e^{-i \theta} \Psi - \Phi_{\omega}), \partial_{\omega} \Phi_{\omega} \rangle_{L^2(\Gamma)} \end{array} \right],
$$
the zeros of which represent the orthogonal constraints in (\ref{orth-decomposition}).

Let $\theta_0$ be the argument in $\inf_{\theta \in \mathbb{R}} \| e^{-i \theta} \Psi - \Phi \|_{H^1(\Gamma)}$
for a given $\Psi \in H^1_{\Gamma}$ satisfying (\ref{orth-given}).
Since the map $\mathbb{R} \ni \omega \mapsto \Phi_{\omega} \in L^2(\Gamma)$ is smooth,
the vector function $G$ is a $C^{\infty}$ map from $\mathbb{R}^2 \times H^1_{\Gamma}$ to $\mathbb{R}^2$.
Thanks to the bound (\ref{orth-given}), there exists a $\delta$-independent constant $C > 0$ such that
$|G(\theta_0,1;\Psi)| \leq C \delta$. Also we obtain
\begin{eqnarray*}
D_{(\theta,\omega)} G(\theta_0,1;\Psi) & = & - \left[ \begin{matrix} 0 &
\langle \Phi, \partial_{\omega} \Phi_{\omega} |_{\omega = 1} \rangle_{L^2(\Gamma)}  \\
\langle \Phi, \partial_{\omega} \Phi_{\omega} |_{\omega = 1} \rangle_{L^2(\Gamma)} & 0 \end{matrix} \right] \\
& \phantom{t} &
+ \left[ \begin{matrix} \langle {\rm Im}(e^{-i \theta_0} \Psi - \Phi), \Phi \rangle_{L^2(\Gamma)} &
\langle {\rm Re}(e^{-i \theta_0} \Psi - \Phi), \partial_{\omega} \Phi_{\omega} |_{\omega = 1}\rangle_{L^2(\Gamma)} \\
- \langle {\rm Re}(e^{-i \theta_0} \Psi - \Phi), \partial_{\omega} \Phi_{\omega} |_{\omega = 1} \rangle_{L^2(\Gamma)}  &
  \langle {\rm Im}(e^{-i \theta_0} \Psi - \Phi), \partial^2_{\omega} \Phi_{\omega} |_{\omega = 1} \rangle_{L^2(\Gamma)} \end{matrix} \right],
\end{eqnarray*}
where $\langle \Phi, \partial_{\omega} \Phi_{\omega} |_{\omega = 1} \rangle_{L^2(\Gamma)} \neq 0$
if $p \in (0,2)$ and the second matrix is bounded by $C \delta$ with a $\delta$-independent constant $C > 0$.
Because the first matrix is invertible if $p \in (0,2)$ and $\delta$ is small, we conclude that
there is $\delta_0 > 0$ such that
$D_{(\theta,\omega)} G(\theta_0,1;\Psi) : \mathbb{R}^2 \to \mathbb{R}^2$ is invertible
with the $\mathcal{O}(1)$ bound on the inverse matrix for every $\delta \in (0,\delta_0)$. By the local inverse mapping
theorem (see Theorem 4.F in \cite{Zeidel}),
for any $\Psi \in H^1_{\Gamma}$ satisfying (\ref{orth-given}),  there exists a unique solution
$(\theta,\omega) \in \mathbb{R}^2$ of the vector equation $G(\theta,\omega;\Psi) = 0$ such
that $|\theta - \theta_0| + |\omega - 1| \leq C \delta$ with a $\delta$-independent constant $C > 0$.
Thus, the bound (\ref{orth-bound}) is satisfied for $\omega$.

By using the definition of $(U,W)$ in the decomposition (\ref{orth-decomposition})
and the triangle inequality for $(\theta,\omega)$ near $(\theta_0,1)$, it is then straightforward to show
that $(U,W)$ are uniquely defined in $H^1_{\Gamma}$ and satisfy the bounds in (\ref{orth-bound}).
\end{proof}

By global well-posedness theory, see Remark \ref{rem-global},
if $\Psi_0 \in H^1_{\Gamma}$, then there exists a unique
solution $\Psi(t) \in C(\mathbb{R},H^1_{\Gamma}) \cap C^1(\mathbb{R},H^{-1}_{\Gamma})$
to the NLS equation (\ref{eq1}) with $p \in (0,2)$ such that $\Psi(0) = \Psi_0$.
For every $\delta > 0$ (sufficiently small), we set
\begin{equation}
\label{initial-data}
\Psi_0 = \Phi + U_0 + i W_0, \quad \| U_0 + i W_0 \|_{H^1(\Gamma)} \leq \delta,
\end{equation}
such that
\begin{equation}
\label{initial-data-constraints}
\langle U_0, \Phi \rangle_{L^2(\Gamma)} = 0, \quad \langle W_0, \partial_{\omega} \Phi_{\omega} |_{\omega = 1} \rangle_{L^2(\Gamma)} = 0.
\end{equation}
Thus, in the initial decomposition (\ref{orth-decomposition}), we choose $\theta_0 = 0$ and $\omega_0 = 1$ at $t = 0$.

\begin{remark}
Compared to the statement of Theorem \ref{theorem-instability}, the initial datum $V := \Psi(0) - \Phi = U_0 + i W_0 \in H^1_{\Gamma}$
is required to satisfy the constraints (\ref{initial-data-constraints}). A more general unstable solution
can be constructed by choosing different initial values for $(\theta_0,\omega_0)$ in the decomposition (\ref{orth-decomposition}).
\end{remark}

Let us assume that $\Psi(t)$ satisfies a priori bound
\begin{equation}
\label{apriori-bound}
\inf_{\theta \in \mathbb{R}} \| e^{-i\theta} \Psi(t) - \Phi \|_{H^1(\Gamma)} \leq \epsilon, \quad t \in [0,t_0],
\end{equation}
for some $t_0 > 0$ and $\epsilon > 0$. This assumption is true at least for small $t_0 > 0$
by the continuity of the global solution $\Psi(t)$. Fix $\epsilon = \delta_0$ defined by
Lemma \ref{lem-orthogonal}. As long as a priori assumption (\ref{apriori-bound}) is satisfied,
Lemma \ref{lem-orthogonal} yields that the unique solution
$\Psi(t)$ to the NLS equation (\ref{eq1}) can be represented as
\begin{equation}
\label{orth-decomposition-time}
\Psi(t) = e^{i \theta(t)} \left[ \Phi_{\omega(t)} + U(t) + i W(t) \right],
\end{equation}
with
\begin{equation}
\label{orth-decomposition-time-constraints}
 \langle U(t), \Phi_{\omega(t)} \rangle_{L^2(\Gamma)} =
\langle W(t), \partial_{\omega} \Phi_{\omega} |_{\omega = \omega(t)} \rangle_{L^2(\Gamma)} = 0.
\end{equation}
Since $\Psi(t) \in C(\mathbb{R},H^1_{\Gamma}) \cap C^1(\mathbb{R},H^{-1}_{\Gamma})$
and the map $\mathbb{R} \ni \omega \mapsto \Phi_{\omega} \in H^1_{\Gamma}$ is smooth, we obtain
$(\theta(t),\omega(t)) \in C^1([0,t_0],\mathbb{R}^2)$, hence
$U(t), W(t) \in C([0,t_0],H^1_{\Gamma}) \cap C^1([0,t_0],H^{-1}_{\Gamma})$.
The proof of Theorem \ref{theorem-instability}
is achieved if we can show that there exists $t_0 > 0$ such that the bound (\ref{apriori-bound})
is true for $t \in [0,t_0]$ but fails as $t > t_0$.

Substituting (\ref{orth-decomposition-time}) into the NLS equation (\ref{eq1}) yields the time evolution
system for the remainder terms:
\begin{equation}
\label{time-evolution}
\frac{d}{dt} \begin{pmatrix} U \\ W \end{pmatrix} =
\begin{pmatrix} 0 & L_-(\omega) \\ -L_+(\omega) & 0 \end{pmatrix} \begin{pmatrix} U \\ W \end{pmatrix}
+ (\dot{\theta} - \omega) \begin{pmatrix} W \\ -(\Phi_{\omega} + U) \end{pmatrix}
- \dot{\omega} \begin{pmatrix} \partial_{\omega} \Phi_{\omega} \\ 0 \end{pmatrix}
+ \begin{pmatrix} -R_U \\ R_W \end{pmatrix},
\end{equation}
where $L_+(\omega) = -\Delta + \omega - (2p+1) (p+1) \Phi_{\omega}^{2p}$,
$L_-(\omega) = -\Delta + \omega - (p+1) \Phi_{\omega}^{2p}$, and
\begin{eqnarray}
\label{RU-term}
R_U & = &  (p+1) \left[ \left((\Phi_{\omega} + U)^2 + W^2 \right)^{p} - \Phi_{\omega}^{2p} \right] W, \\
R_W & = &  (p+1) \left[ \left( (\Phi_{\omega} + U)^2+W^2 \right)^{p} \left(\Phi_{\omega}+U\right) -
\Phi_{\omega}^{2p} \left( \Phi_{\omega}+U \right) - 2p\Phi_{\omega}^{2p} U \right].
\label{RW-term}
\end{eqnarray}
By using the symplectically orthogonal conditions (\ref{orth-decomposition-time-constraints})
in the decomposition (\ref{orth-decomposition-time}),
we obtain from system (\ref{time-evolution}) the modulation equations for parameters $(\theta,\omega)$:
\begin{equation}
\label{time-evolution-ode}
\begin{pmatrix} \langle \Phi_{\omega}, W \rangle_{L^2(\Gamma)} & -\langle \partial_{\omega} \Phi_{\omega}, \Phi_{\omega} - U \rangle_{L^2(\Gamma)} \\
\langle \partial_{\omega} \Phi_{\omega}, \Phi_{\omega} + U  \rangle_{L^2(\Gamma)} &
-\langle \partial_{\omega}^2 \Phi_{\omega}, W \rangle_{L^2(\Gamma)} \end{pmatrix} \begin{pmatrix} \dot{\theta} - \omega \\ \dot{\omega} \end{pmatrix}
= \begin{pmatrix} \langle \Phi_{\omega}, R_U \rangle_{L^2(\Gamma)} \\ \langle \partial_{\omega} \Phi_{\omega}, R_W \rangle_{L^2(\Gamma)} \end{pmatrix}.
\end{equation}

The modulation equations (\ref{time-evolution-ode}) and the time-evolution system (\ref{time-evolution}) have been
studied in many contexts involving dynamics of solitary waves \cite{Cuc,Miz,PelPhan,Soffer}. In the context of orbital
instability of the half-soliton states, we are able to avoid detailed analysis of system (\ref{time-evolution}) and (\ref{time-evolution-ode})
by using conservation of the energy $E$ and mass $Q$ defined by (\ref{energy}). The following result provide some estimates
on the derivatives of the modulation parameters $\theta$ and $\omega$.

\begin{lemma}
\label{lem-parameters}
Assume that $\omega \in \mathbb{R}$ and $U, W \in H^1_{\Gamma}$ satisfy
\begin{equation}
\label{bounds-parameters-apriori}
|\omega - 1 | + \| U + i W \|_{H^1(\Gamma)} \leq \epsilon
\end{equation}
for sufficiently small $\epsilon > 0$. For every $p \in \left[ \frac{1}{2}, 2 \right)$, there exists
an $\epsilon$-independent constant $A > 0$ such that
\begin{equation}
\label{bounds-parameters}
| \dot{\theta} - \omega | \leq A \left( \| U \|^2_{H^1(\Gamma)} + \| W \|^2_{H^1(\Gamma)} \right),
\quad | \dot{\omega} | \leq A \| U \|_{H^1(\Gamma)} \| W \|_{H^1(\Gamma)}.
\end{equation}
\end{lemma}

\begin{proof}
If $\langle \Phi_{\omega}, \partial_{\omega} \Phi_{\omega} \rangle_{L^2(\Gamma)} \neq 0$ for $p \neq 2$
and under assumption (\ref{bounds-parameters-apriori}), the coefficient matrix of
system (\ref{time-evolution-ode}) is invertible with the $\mathcal{O}(1)$ bound
on the inverse matrix for sufficiently small $\epsilon > 0$.
For every $p \geq \frac{1}{2}$, the Taylor expansion of the nonlinear functions $R_U$ and $R_W$
in (\ref{RU-term}) and (\ref{RW-term}) yield
\begin{equation}
\label{RU-term-leading}
R_U = 2p(p+1) \Phi_{\omega}^{2p-1} U W + \tilde{R}_U
\end{equation}
and
\begin{equation}
\label{RW-term-leading}
R_W = p(p+1) \Phi_{\omega}^{2p-1} \left[ (2p+1) U^2 + W^2 \right] + \tilde{R}_W,
\end{equation}
where $\tilde{R}_U$ and $\tilde{R}_W$ satisfies
\begin{eqnarray*}
\| \tilde{R}_U \|_{H^1(\Gamma)} + \| \tilde{R}_W \|_{H^1(\Gamma)} =
\left\{ \begin{array}{l} {\rm o}(\| U + i W \|_{H^1(\Gamma)}^2), \quad p \in \left(\frac{1}{2},1\right), \\
{\rm O}(\| U + i W \|_{H^1(\Gamma)}^3), \quad p \geq 1. \end{array} \right.
\end{eqnarray*}
The leading-order terms in (\ref{RU-term-leading})--(\ref{RW-term-leading}) and the Banach algebra property
of $H^1(\Gamma)$ yield the bound (\ref{bounds-parameters}).
\end{proof}

\subsection{Step 2: Symplectic projections to the neutral modes}

Let us recall the orthogonal basis of eigenvectors constructed algorithmically in Remark \ref{rem-vectors}.
We denote the corresponding invariant subspace by $X_c := {\rm span}\{ U^{(1)}, U^{(2)}, \dots, U^{(N-1)}\}$.
For each vector $U^{(j)}$ with $1 \leq j \leq N-1$, we construct the generalized vector $W^{(j)}$ from
solutions of the linear system $L_- W^{(j)} = U^{(j)}$, which exists uniquely in $L^2_c$
thanks to the fact that $U^{(j)} \in L^2_c$ in (\ref{constraint}) and ${\rm ker}(L_-) = {\rm span}\{\Phi\}$.
Let us denote the corresponding invariant subspace by
$X_c^* := {\rm span}\{W^{(1)},W^{(2)},\dots,W^{(N-1)}\}$.

\begin{lemma}
\label{lem-gen-eigenvectors}
Basis vectors in $X_c$ and $X_c^*$ are symplectically orthogonal in the sense
\begin{equation}
\label{orthogonality-U-W}
\langle U^{(j)}, W^{(k)} \rangle_{L^2(\Gamma)} = 0, \quad j \neq k \quad {\rm and} \quad
\langle U^{(j)}, W^{(j)} \rangle_{L^2(\Gamma)} > 0.
\end{equation}
Moreover, basis vectors are also orthogonal to each other.
\end{lemma}

\begin{proof}
Let us represent $U^{(j)}$ by
$$
U^{(j)}(x) = \phi'(x) e_j,
$$
where $e_j \in \mathbb{R}^N$ is $x$-independent and $\phi(x) = {\rm sech}^{\frac{1}{p}}(px)$.
Then $W^{(j)}$ can be represented by the explicit expression
$$
W^{(j)}(x) = -\frac{1}{2} x \phi(x) e_j.
$$
Since $\{ e_1,e_2,\dots, e_{N-1} \}$ are orthogonal by the construction,
the set  $\{ W^{(1)}, W^{(2)}, \dots, W^{(N-1)}\}$ is also orthogonal.
Moreover, it is orthogonal to the set $\{ U^{(1)}, U^{(2)}, \dots, U^{(N-1)}\}$. It follows from
the explicit computation
\begin{equation}
\label{normalization-U-W}
\langle U^{(j)}, W^{(j)} \rangle_{L^2(\Gamma)} = \frac{1}{4} \| \phi \|_{L^2(\mathbb{R}_+)}^2 \| e_j \|^2
\end{equation}
that $\langle U^{(j)}, W^{(j)} \rangle_{L^2(\Gamma)} > 0$ for each $j$.
Thus, (\ref{orthogonality-U-W}) is proved.
\end{proof}

Although the coercivity of $L_+$ was only proved with respect to the bases in $X_c$, see Lemma \ref{prop-coercivity-L-plus},
the result can now be transferred to the symplectically dual basis.

\begin{lemma}
\label{lem-coercivity}
For every $p \in (0,2)$, there exists $C_p > 0$ such that
\begin{equation}
\label{coercivity-L-plus-dual}
\langle L_+ U, U \rangle_{L^2(\Gamma)} \geq C_p \| U \|_{H^1(\Gamma)}^2 \quad \mbox{\rm for every \;} U \in H^1_{\Gamma} \cap L^2_c \cap [X_c^*]^{\perp}.
\end{equation}
\end{lemma}

\begin{proof}
It follows from Lemma \ref{prop-coercivity-L-plus} that $\langle L_+ U, U \rangle_{L^2(\Gamma)} \geq 0$ for $p \in (0,2)$
if $U \in H^1_{\Gamma} \cap L^2_c$. Moreover, $\langle L_+ U, U \rangle_{L^2(\Gamma)} = 0$ if and only if
$U \in X_c$. Thanks to the orthogonality and positivity of diagonal terms in the symplectically dual bases in $X_c$ and $X_c^*$,
see (\ref{orthogonality-U-W}), the coercivity bound (\ref{coercivity-L-plus-dual})
follows from the bound (\ref{coercivity-L-plus}) by the standard variational principle.
\end{proof}

Similarly, the coercivity of $L_-$ was proved with respect to the constraint in $L^2_c$, see Lemma \ref{prop-coercivity-L-minus}.
The following lemma transfers the result to the symplectically dual constraint.

\begin{lemma}
\label{lem-coercivity-L-minus}
For every $p \in (0,2)$, there exists $C_p > 0$ such that
\begin{equation}
\label{coercivity-minus}
\langle L_- W, W \rangle_{L^2(\Gamma)} \geq C \| W \|_{H^1(\Gamma)}^2 \quad \mbox{\rm for every \;} W \in H^1_{\Gamma} \cap (L^2_c)^*,
\end{equation}
where $(L^2_c)^* = \{ W \in L^2(\Gamma) : \;\; \langle W, \partial_{\omega} \Phi_{\omega} |_{\omega = 1} \rangle_{L^2(\Gamma)} = 0 \}$.
\end{lemma}

\begin{proof}
It follows from Lemma \ref{prop-coercivity-L-minus} that $\langle L_- W, W \rangle_{L^2(\Gamma)} \geq 0$
if $W \in H^1_{\Gamma}$. Moreover, $\langle L_- W, W \rangle_{L^2(\Gamma)} = 0$ if and only if
$W \in {\rm span}(\Phi)$. Thanks to the positivity
$\langle \partial_{\omega} \Phi_{\omega} |_{\omega = 1}, \Phi \rangle_{L^2(\Gamma)} > 0$ in (\ref{positivity-Phi})
and (\ref{norm-squared}) for $p \in (0,2)$, the coercivity bound (\ref{coercivity-minus}) follows from
the bound (\ref{coercivity}) by the standard variational principle.
\end{proof}

\begin{remark}
By using the scaling transformation (\ref{scaling-transform}), we can continue the basis vectors for $\omega \neq 1$.
For notational convenience, $\omega$ is added as a subscript if the expressions are
continued with respect to $\omega$.
\end{remark}

Recall the symplectically orthogonal decomposition of the unique solution
$\Psi(t)$ to the NLS equation (\ref{eq1}) in the form (\ref{orth-decomposition-time})--(\ref{orth-decomposition-time-constraints}).
Let us further decompose the remainder terms $U(t)$ and $W(t)$ in (\ref{orth-decomposition-time})
over the orthogonal bases in $X_c$ and $X_c^*$,
which are also symplectically orthogonal to each other by Lemma \ref{lem-gen-eigenvectors}.
More precisely, since $\omega(t)$ changes we set
\begin{equation}\label{r1}
U(t) = \sum_{j=1}^{N-1} c_j(t) U_{\omega(t)}^{(j)} + U^{\perp}(t), \quad
W(t) = \sum_{j=1}^{N-1} b_j(t) W_{\omega(t)}^{(j)} + W^{\perp}(t),
\end{equation}
and require
\begin{equation}
\label{orth-decomposition-time-second}
\langle U^{\perp}(t), W^{(j)}_{\omega(t)} \rangle_{L^2(\Gamma)} =
\langle W^{\perp}(t), U^{(j)}_{\omega(t)} \rangle_{L^2(\Gamma)} = 0, \quad 1 \leq j \leq N-1.
\end{equation}
Since $\{ \langle U^{(j)}_{\omega}, W^{(k)}_{\omega} \rangle_{L^2(\Gamma)} \}_{1 \leq j,k \leq N-1}$ is a positive diagonal matrix
by the conditions (\ref{orthogonality-U-W}),
the projections $c = (c_1,c_2,\dots,c_{N-1}) \in \mathbb{R}^{N-1}$ and $b = (b_1,b_2,\dots,b_{N-1}) \in \mathbb{R}^{N-1}$
in (\ref{r1}) are uniquely determined by $U$ and $W$
and so are the remainder terms $U^{\perp}$ and $W^{\perp}$.
Because $\omega(t) \in C^1([0,t_0],\mathbb{R})$ and
$U(t), W(t) \in C([0,t_0],H^1_{\Gamma}) \cap C^1([0,t_0],H^{-1}_{\Gamma})$, we
have $c(t), b(t) \in C^1([0,t_0],\mathbb{R}^{N-1})$ and
$U^{\perp}(t), W^{\perp}(t) \in C([0,t_0],H^1_{\Gamma}) \cap C^1([0,t_0],H^{-1}_{\Gamma})$.

When the decomposition \eqref{r1} is substituted to the time evolution problem (\ref{time-evolution}),
we obtain
\begin{equation}
\label{Vbot}
\frac{d U^{\perp}}{dt} + \sum_{j =1}^{N-1} \left[ \frac{dc_j}{dt} - b_j \right] U_{\omega}^{(j)} =
L_-(\omega) W^{\perp}
+ (\dot{\theta} - \omega) W
- \dot{\omega}  \left[  \partial_{\omega} \Phi_{\omega} + \sum_{j=1}^{N-1} c_j(t) \partial_{\omega} U_{\omega}^{(j)} \right]
-R_U
\end{equation}
and
\begin{equation}\label{Wbot}
\frac{d W^{\perp}}{dt} + \sum_{j =1}^{N-1} \frac{db_j}{dt} W_{\omega}^{(j)} =
-L_+(\omega) U^{\perp}
- (\dot{\theta} - \omega) \left[ \Phi_{\omega} + U \right]
- \dot{\omega}  \sum_{j=1}^{N-1} b_j(t) \partial_{\omega} W_{\omega}^{(j)} +R_W,
\end{equation}
where $R_U$ and $R_W$ are rewritten from (\ref{RU-term}) and (\ref{RW-term}) after
$U$ and $W$ are expressed by (\ref{r1}).

By using symplectically orthogonal projections (\ref{orth-decomposition-time-second}),
we obtain from (\ref{Vbot}) and (\ref{Wbot}) a system of differential equations for
the amplitudes $(c_j,b_j)$ for every $1 \leq j \leq N-1$:
\begin{eqnarray}
\label{system-U-W}
\langle W_{\omega}^{(j)},U^{(j)}_{\omega}  \rangle_{L^2(\Gamma)}
\left[ \frac{dc_j}{dt} - b_j \right] = R_c^{(j)}, \quad
\langle W_{\omega}^{(j)},U^{(j)}_{\omega}  \rangle_{L^2(\Gamma)}
\frac{db_j}{dt} = R_b^{(j)},
\end{eqnarray}
where
\begin{eqnarray*}
R_c^{(j)} & = &  \dot{\omega} \left[ \langle \partial_{\omega} W_{\omega}^{(j)}, U^{\perp} \rangle_{L^2(\Gamma)}
-  \sum_{i=1}^{N-1} c_i \langle W_{\omega}^{(j)}, \partial_{\omega} U^{(i)}_{\omega}  \rangle_{L^2(\Gamma)} \right] \\
& \phantom{t} & + (\dot{\theta}-\omega) \langle W_{\omega}^{(j)}, W \rangle_{L^2(\Gamma)}
- \langle W_{\omega}^{(j)}, R_U \rangle_{L^2(\Gamma)}, \\
R_b^{(j)} & = &  \dot{\omega} \left[ \langle \partial_{\omega} U_{\omega}^{(j)}, W^{\perp} \rangle_{L^2(\Gamma)}
-  \sum_{i=1}^{N-1} b_i \langle U_{\omega}^{(j)}, \partial_{\omega} W^{(i)}_{\omega}  \rangle_{L^2(\Gamma)} \right] \\
& \phantom{t} & - (\dot{\theta}-\omega) \langle U_{\omega}^{(j)}, U \rangle_{L^2(\Gamma)}
+ \langle U_{\omega}^{(j)}, R_W \rangle_{L^2(\Gamma)},
\end{eqnarray*}
and we have used the orthogonality conditions:
$$
\langle U_{\omega}^{(j)}, \Phi_{\omega} \rangle_{L^2(\Gamma)} =
\langle W_{\omega}^{(j)}, \partial_{\omega} \Phi_{\omega}  \rangle_{L^2(\Gamma)} = 0, \quad 1 \leq j \leq N-1.
$$
The terms $\dot{\omega}$ and $\dot{\theta} - \omega$ can be expressed from the system (\ref{time-evolution-ode}),
where $U$ and $W$ are again expressed by (\ref{orth-decomposition-time-second}).

\subsection{Step 3: Truncated Hamiltonian system of $(N-1)$ degrees of freedom}

The truncated Hamiltonian system of $(N-1)$ degrees of freedom follows
from the formal truncation of system (\ref{system-U-W}) with $\omega = 1$
at the leading order:
\begin{eqnarray}
\label{normal-form-time} \phantom{texttext}
\left\{ \begin{array}{l} \dot{\gamma}_j =  \beta_j, \\
\langle W^{(j)},U^{(j)} \rangle_{L^2(\Gamma)}
\dot{\beta}_j =  p(p+1) (2p+1) \sum\limits_{k=1}^{N-1} \sum\limits_{n = 1}^{N-1}
\langle \Phi^{2p-1} U^{(k)} U^{(n)}, U^{(j)} \rangle_{L^2(\Gamma)} \gamma_k \gamma_n.\end{array} \right.
\end{eqnarray}
By using the function $M_0(\gamma)$ given by (\ref{energy-cube}),
we can write the truncated system (\ref{normal-form-time}) in the Hamiltonian form
\begin{equation}
\label{normal-form-Ham}
\left\{ \begin{array}{l}
2 \langle W^{(j)},U^{(j)} \rangle_{L^2(\Gamma)} \dot{\gamma}_j =  \partial_{\beta_j} H_0(\gamma,\beta), \\
2 \langle W^{(j)},U^{(j)} \rangle_{L^2(\Gamma)} \dot{\beta}_j =  - \partial_{\gamma_j} H_0(\gamma,\beta),\end{array} \right.
\end{equation}
which is generated by the Hamiltonian
\begin{equation}
\label{Hamiltonian-0}
H_0(\gamma,\beta) := \sum_{j=1}^{N-1} \langle W^{(j)},U^{(j)}  \rangle_{L^2(\Gamma)} \beta_j^2 + M_0(\gamma).
\end{equation}
The reduced Hamiltonian $H_0$ arises naturally in the expansion of the action functional
$\Lambda$. The following result implies nonlinear instability of the zero equilibrium point
in the finite-dimensional Hamiltonian system (\ref{normal-form-Ham})--(\ref{Hamiltonian-0}).

\begin{lemma}
\label{lem-instability-Ham}
There exists $\epsilon > 0$ such that for every $\delta > 0$ (sufficiently small), there is an initial point
$(\gamma(0),\beta(0))$ with $\| \gamma(0)\| + \| \beta(0) \| \leq \delta$ such that
the unique solution of the finite-dimensional system (\ref{normal-form-time})
satisfies $\| \gamma(t_0) \| > \epsilon$ for some $t_0 = \mathcal{O}(\epsilon^{-1/2})$.
\end{lemma}

\begin{proof}
We claim that $\gamma_1 = \gamma_2 = \dots = \gamma_{N-2} =0$ is an invariant reduction of system (\ref{normal-form-time}).
In order to show this, we compute coefficients of the function $M_0(\gamma)$ in (\ref{energy-cube}) that contains
$\gamma_i \gamma_j \gamma_{N-1}$ for $i,j \neq N-1$:
$$
\langle \Phi^{2p-1} U^{(i)} U^{(j)}, U^{(N-1)} \rangle_{L^2(\Gamma)} = \langle e_i, e_j \rangle \int_0^{\infty} \phi^{2p-1} (\phi')^3 dx
$$
Since $\langle e_i, e_j \rangle = 0$ for every $i \neq j$, the function $M_0(\gamma)$ depends on $\gamma_{N-1}$
only in the terms $\gamma_1^2 \gamma_{N-1}$, $\gamma_2^2 \gamma_{N-1}$, $\dots$, $\gamma_{N-2}^2 \gamma_{N-1}$, as well as
$\gamma_{N-1}^3$. Therefore, $\gamma_1 = \gamma_2 = \dots = \gamma_{N-2} =0$ is an invariant solution
of the first $(N-2)$ equations of system (\ref{normal-form-time}). The last equation yields
the following second-order differential equation for $\gamma_{N-1}$:
\begin{equation}
\label{scalar-ODE}
\langle W^{(N-1)},U^{(N-1)} \rangle_{L^2(\Gamma)}
\ddot{\gamma}_{N-1} =  p(p+1) (2p+1) \langle \Phi^{2p-1} U^{(N-1)} U^{(N-1)}, U^{(N-1)} \rangle_{L^2(\Gamma)} \gamma_{N-1}^2,
\end{equation}
where the coefficient is nonzero thanks to (\ref{energy-cube-nonzero}) and (\ref{normalization-U-W}). Since the zero
equilibrium is unstable in the scalar equation (\ref{scalar-ODE}), it is then unstable in system (\ref{normal-form-time}).
If $\gamma(t) = \mathcal{O}(\epsilon)$ for $t \in [0,t_0]$, then $\epsilon^2 t_0^2 = \mathcal{O}(\epsilon)$,
hence the nonlinear instability develops at the time span $[0,t_0]$ with $t_0 = \mathcal{O}(\epsilon^{-1/2})$.
\end{proof}

\begin{remark}
For $N = 3$, we have $M_0(\gamma) = 2p^2 (\gamma_1^2 - \gamma_2^2) \gamma_2$.
Computing the normalization conditions (\ref{normalization-U-W}),
we obtain the following finite-dimensional system of degree two:
\begin{eqnarray}
\label{normal-form-3}
\left\{ \begin{array}{l}
\| \phi \|^2_{L^2(\mathbb{R}_+)} \ddot{\gamma}_1 = -4 p^2 \gamma_1 \gamma_2,\\
3 \| \phi \|^2_{L^2(\mathbb{R}_+)} \ddot{\gamma}_2 = -2 p^2 (\gamma_1^2-3 \gamma_2^2).
\end{array} \right.
\end{eqnarray}
For $N = 4$, we have $M_0(\gamma) = 2p^2 ( \gamma_1^2 \gamma_2
+ \gamma_1^2 \gamma_3 - \gamma_2^3 + 3 \gamma_2^2 \gamma_3 - 4 \gamma_3^3)$.
Computing the normalization conditions (\ref{normalization-U-W}),
we obtain the following finite-dimensional system of degree three:
\begin{eqnarray}
\label{normal-form-4}
\left\{ \begin{array}{l}
\| \phi \|^2_{L^2(\mathbb{R}_+)} \ddot{\gamma}_1 = -4 p^2 \gamma_1 (\gamma_2+ \gamma_3),\\
3 \| \phi \|^2_{L^2(\mathbb{R}_+)} \ddot{\gamma}_2 = -2 p^2 (\gamma_1^2-3 \gamma_2^2 + 6 \gamma_2 \gamma_3), \\
3 \| \phi \|^2_{L^2(\mathbb{R}_+)} \ddot{\gamma}_3 = -p^2 (\gamma_1^2+3 \gamma_2^2 - 12 \gamma_3^2).
\end{array} \right.
\end{eqnarray}
\end{remark}

\subsection{Step 4: Expansion of the action functional}

Recall the action functional $\Lambda(\Psi) = E(\Psi) + Q(\Psi)$, for which $\Phi$ is a critical point.
By using the scaling transformation (\ref{scaling-transform}),
we continue the action functional for $\omega \neq 1$ and define the following function:
\begin{equation}
\label{energy-difference}
\Delta(t) := E(\Phi_{\omega(t)} + U(t) + i W(t)) - E(\Phi) + \omega(t) \left[ Q(\Phi_{\omega(t)} + U(t) + i W(t)) - Q(\Phi) \right].
\end{equation}
As long as a priori bound (\ref{apriori-bound}) is satisfied, one can expand $\Delta$
by using the primary decomposition (\ref{orth-decomposition-time}) as follows:
\begin{equation}
\label{energy-expansion}
\Delta = D(\omega) + \langle L_+(\omega) U, U \rangle_{L^2(\Gamma)} + \langle L_-(\omega) W, W \rangle_{L^2(\Gamma)} + N_{\omega}(U,W),
\end{equation}
where the dependence of all quantities on $t$ is ignored, $D(\omega)$ is defined by
$$
D(\omega) := E(\Phi_{\omega}) - E(\Phi) + \omega \left[ Q(\Phi_{\omega})  - Q(\Phi) \right],
$$
and
\begin{eqnarray*}
N_{\omega}(U,W) = \left\{ \begin{array}{l} {\rm o}(\| U + i W \|_{H^1(\Gamma)}^2), \quad p \in \left(0,\frac{1}{2}\right), \\
{\rm O}(\| U + i W \|_{H^1(\Gamma)}^3), \quad p \geq \frac{1}{2}, \end{array} \right.
\end{eqnarray*}
is a continuation of $N(U,W)$ defined by (\ref{second-variation-Lambda}) with respect to $\omega$.

Since $D'(\omega) = Q(\Phi_{\omega}) - Q(\Phi)$
thanks to the variational characterization of $\Phi_{\omega}$, we have $D(1) = D'(1) = 0$, and
\begin{equation}
\label{energy-0}
D(\omega) = (\omega - 1)^2 \langle \Phi, \partial_{\omega} \Phi_{\omega} |_{\omega = 1} \rangle_{L^2(\Omega)}
+ \tilde{D}(\omega),
\end{equation}
where $\tilde{D}(\omega) = \mathcal{O}(|\omega - 1|^3)$. Thanks to conservation
of the energy $E$ and mass $Q$ defined by (\ref{energy})
and to the phase invariance in the NLS, we represent $\Delta(t)$ in terms of the initial data
$\omega(0) = \omega_0 = 1$, $U(0) = U_0$, and $W(0) = W_0$ as follows:
\begin{equation}
\label{energy-1}
\Delta(t) = \Delta_0 + \left( \omega(t) - 1 \right) \left[ Q(\Phi + U_0 + i W_0) - Q(\Phi)\right],
\end{equation}
where
\begin{equation}
\label{energy-2}
\Delta_0 := E(\Phi + U_0 + i W_0) - E(\Phi) + Q(\Phi + U_0 + i W_0) - Q(\Phi)
\end{equation}
is a constant of motion.

Let us now consider the secondary decomposition (\ref{r1})--(\ref{orth-decomposition-time-second}).
If the solution given by (\ref{orth-decomposition-time}) and (\ref{r1}) satisfies a priori bound (\ref{apriori-bound})
for some $t_0 > 0$ and $\epsilon > 0$, then the coefficients of the secondary decomposition (\ref{r1}) are required
to satisfy the bound
\begin{equation}
\label{apriori-bound-c-b}
|\omega(t) - 1 | + \| c(t) \| + \| b(t) \| + \| U^{\perp}(t)+ i W^{\perp}(t) \|_{H^1(\Gamma)} \leq A \epsilon,
\quad t \in [0,t_0],
\end{equation}
for an $\epsilon$-independent constant $A > 0$. We substitute the secondary decomposition (\ref{r1})--(\ref{orth-decomposition-time-second}) into
the representation (\ref{energy-expansion}) and estimate the corresponding expansion.

\begin{lemma}
Assume that $\omega \in \mathbb{R}$, $c,b \in \mathbb{R}^{N-1}$, and $U^{\perp},W^{\perp} \in H^1_{\Gamma}$ satisfy
the bound (\ref{apriori-bound-c-b}) for sufficiently small $\epsilon > 0$.
For every $p \geq \frac{1}{2}$, there exists an $\epsilon$-independent constant $A > 0$ such that
the representation (\ref{energy-expansion}) is expanded as follows:
\begin{eqnarray}
\nonumber
\Delta & = & D(\omega) + \langle L_+(\omega) U^{\perp}, U^{\perp} \rangle_{L^2(\Gamma)} + \langle L_-(\omega) W^{\perp}, W^{\perp}
\rangle_{L^2(\Gamma)} \\
& \phantom{t} & + \sum_{j=1}^{N-1} \langle W_{\omega}^{(j)},U^{(j)}_{\omega}  \rangle_{L^2(\Gamma)} b_j^2
+ M_0(c) + \widetilde{\Delta}(c,b,U^{\perp},W^{\perp}),
\label{energy-expansion-second}
\end{eqnarray}
with
\begin{eqnarray}
\nonumber
| \widetilde{\Delta}(c,b,U^{\perp},W^{\perp}) | & \leq & A \left( \mu(\| c \|) +
\| c \|^2  \| U^{\perp} \|_{H^1(\Gamma)} + \| U^{\perp} \|_{H^1(\Gamma)}^3 + \| c \| \| b \|^2  \right. \\
& \phantom{t} & \left. + \| c \| \| W^{\perp} \|_{H^1(\Gamma)}^2
+ \| b \|^2 \| U^{\perp} \|_{H^1(\Gamma)} +  \| U^{\perp} \|_{H^1(\Gamma)} \| W^{\perp} \|_{H^1(\Gamma)}^2 \right),
\label{energy-estimate}
\end{eqnarray}
where $M_0(c)$ is given by (\ref{energy-cube}) and
\begin{equation}
\label{mu-remainder}
\mu(\| c \|) = \left\{ \begin{array}{l} {\rm o}(\| c \|^3), \quad p \in \left(\frac{1}{2},1\right), \\
{\rm O}(\| c\|^4), \quad p \geq 1. \end{array} \right.
\end{equation}
\end{lemma}

\begin{proof}
For every $p \geq \frac{1}{2}$, Taylor expansion of $N_{\omega}(U,W)$ yields
{\small\begin{eqnarray*}
N_{\omega}(U,W) = - \frac{2}{3} p(p+1)(2p+1)\langle \Phi^{2p-1} U^2, U \rangle_{L^2(\Gamma)}
- 2 p(p+1) \langle \Phi^{2p-1} W^2, U \rangle_{L^2(\Gamma)} + S_{\omega}(U,W),\label{N-expansion}
\end{eqnarray*}}
where
\begin{eqnarray*}
S_{\omega}(U,W) = \left\{ \begin{array}{l} {\rm o}(\| U + i W \|_{H^1(\Gamma)}^3), \quad p \in \left(\frac{1}{2},1\right), \\
{\rm O}(\| U + i W \|_{H^1(\Gamma)}^4), \quad p \geq 1. \end{array} \right.
\end{eqnarray*}
is a continuation of $S(U,W)$ defined by (\ref{Lambda-nonlinear}) with respect to $\omega$.
The expansion (\ref{energy-expansion-second}) holds by substituting of (\ref{r1}) into (\ref{energy-expansion})
and estimating the remainder terms thanks to Banach algebra property of $H^1(\Gamma)$ and
the assumption (\ref{apriori-bound-c-b}). Only the end-point bounds
are incorporated into the estimate (\ref{energy-estimate}).
\end{proof}

We bring (\ref{energy-1}) and (\ref{energy-expansion-second}) together as follows:
\begin{eqnarray}
\nonumber
\Delta_0 - H_0(c,b) & = & D(\omega) - \left( \omega - 1 \right) \left[ Q(\Phi + U_0 + i W_0) - Q(\Phi)\right] \\
& \phantom{t} & + \langle L_+(\omega) U^{\perp}, U^{\perp} \rangle_{L^2(\Gamma)} + \langle L_-(\omega) W^{\perp}, W^{\perp}
\rangle_{L^2(\Gamma)} + \widetilde{\Delta}(c,b,U^{\perp},W^{\perp}),
\label{energy-3}
\end{eqnarray}
where $H_0(c,b)$ is given by (\ref{Hamiltonian-0}).
Recall that the energy $E(\Psi)$ and mass $Q(\Psi)$ are bounded in $H^1_{\Gamma}$,
whereas $\Phi$ is a critical point of $E$ under fixed $Q$. Thanks to the bound (\ref{initial-data}) on the initial data,
the orthogonality (\ref{initial-data-constraints}), and
the representation (\ref{energy-2}), there is an $\delta$-independent constant $A > 0$ such that
\begin{equation}
\label{bound-Delta-0}
|\Delta_0 | + |Q(\Phi + U_0 + i W_0) - Q(\Phi)| \leq A \delta^2.
\end{equation}
Thanks to the representations (\ref{energy-cube}) and (\ref{Hamiltonian-0}),
there is a generic constant $A > 0$ such that
\begin{equation}
\label{bound-H-0}
|H_0(c,b) | \leq A \left( \| c \|^3 + \| b \|^2 \right).
\end{equation}
The value of $\omega$ near $\omega_0 = 1$ and the remainder terms $U^{\perp},W^{\perp}$ in the $H^1(\Gamma)$ norm
can be controlled in the time evolution of the NLS equation (\ref{eq1}) by using the energy expansion (\ref{energy-3}).
The following lemma presents this result.

\begin{lemma}
\label{lem-remainder-last}
Consider a solution to the NLS with $p \geq \frac{1}{2}$ given by (\ref{orth-decomposition-time}) and (\ref{r1})
with $\omega(t) \in C^1([0,t_0],\mathbb{R})$,
$c(t),b(t) \in C^1([0,t_0],\mathbb{R}^{N-1})$, and $U^{\perp}(t), W^{\perp}(t) \in C([0,t_0],H^1_{\Gamma})$
satisfying the bound (\ref{apriori-bound-c-b}) for sufficiently small $\epsilon > 0$.
Then, there exists an $\epsilon$-independent constant $A > 0$ such that for every $t \in [0,t_0]$,
\begin{equation}
\label{bound-omega-U-W}
|\omega - 1|^2 +  \| U^{\perp} + i W^{\perp} \|_{H^1(\Gamma)}^2 \leq A
 \left[ \delta^2 + |H_0(c,b)| +  \mu(\| c \|) + \| c \| \| b \|^2 + \| b \|^3 \right],
\end{equation}
where $\mu(\| c \|)$ is the same as in (\ref{mu-remainder}).
\end{lemma}

\begin{proof}
The bound on $|\omega-1|^2$ follows from (\ref{energy-0}), (\ref{energy-estimate}), (\ref{energy-3}), and (\ref{bound-Delta-0})
thanks to the positivity of $D''(1) = 2 \langle \Phi, \partial_{\omega} \Phi_{\omega} |_{\omega = 1} \rangle_{L^2(\Gamma)}$.
The bounds on $\| U^{\perp} \|_{H^1(\Gamma)}^2$ and $\| W^{\perp} \|_{H^1(\Gamma)}^2$ follow
from (\ref{energy-estimate}), (\ref{energy-3}), and (\ref{bound-Delta-0})
thanks to the coercivity of $L_+(\omega)$ and $L_-(\omega)$ in Lemmas \ref{lem-coercivity} and \ref{lem-coercivity-L-minus}.
\end{proof}

\subsection{Step 5: Closing the energy estimates}

By Lemma \ref{lem-instability-Ham}, there exists a trajectory of the finite-dimensional system
(\ref{normal-form-time}) near the zero equilibrium which leaves the $\epsilon$-neighborhood
of the zero equilibrium. This nonlinear instability developes over the time span $[0,t_0]$ with
$t_0 = \mathcal{O}(\epsilon^{-1/2})$.
The second equation of system (\ref{normal-form-time}) shows that if $\gamma(t) = \mathcal{O}(\epsilon)$
for $t \in [0,t_0]$ and $t_0 = \mathcal{O}(\epsilon^{-1/2})$, then $\beta(t) = \mathcal{O}(\epsilon^{3/2})$
for $t \in [0,t_0]$. It is also clear that the scaling above is consistent with the first equation
of system (\ref{normal-form-time}). The scaling above suggests to consider the following region in the phase space
$\mathbb{R}^{N-1} \times \mathbb{R}^{N-1}$:
\begin{equation}
\label{bound-coefficients-apriori}
\| c(t) \| \leq A \epsilon, \quad \| b(t) \| \leq A \epsilon^{3/2}, \quad t \in [0,t_0], \quad t_0 \leq A \epsilon^{-1/2},
\end{equation}
for an $\epsilon$-independent constant $A > 0$.
The region in (\ref{bound-coefficients-apriori}) satisfies a priori assumption (\ref{apriori-bound-c-b}) for $c$ and $b$.
The following result shows that a trajectory
of the full system (\ref{system-U-W}) follows closely to the trajectory of the finite-dimensional system (\ref{normal-form-time})
in the region (\ref{bound-coefficients-apriori}).

\begin{lemma}
\label{lem-correction}
Consider a solution $\gamma(t),\beta(t) \in C^1([0,t_0],\mathbb{R}^{N-1})$ to
the finite-dimensional system (\ref{normal-form-time}) in the region (\ref{bound-coefficients-apriori})
with sufficiently small $\epsilon > 0$.
Then, a solution $c(t),b(t) \in C^1([0,t_0],\mathbb{R}^{N-1})$ to system (\ref{system-U-W})
remains in the region (\ref{bound-coefficients-apriori}) and there exist an $\epsilon$-independent constant
$A > 0$ such that
\begin{equation}
\label{bound-coefficients}
\| c(t) - \gamma(t) \| \leq A \nu(\epsilon), \quad \| b(t) - \beta(t) \| \leq A \epsilon^{1/2} \nu(\epsilon), \quad t \in [0,t_0],
\end{equation}
where
\begin{equation}
\label{nu}
\nu(\epsilon) = \left\{ \begin{array}{l} {\rm o}(\epsilon), \quad p \in \left(\frac{1}{2},1\right), \\
{\rm O}(\epsilon^{3/2}), \quad p \geq 1. \end{array} \right.
\end{equation}
\end{lemma}

\begin{proof}
By the bounds (\ref{bound-H-0}) and (\ref{bound-omega-U-W}),
as well as a priori assumption (\ref{bound-coefficients-apriori}),
there exists an $(\delta,\epsilon)$-independent constant $A > 0$ such that
\begin{equation}
\label{bound-omega-U-W-new}
|\omega(t) - 1| +  \| U^{\perp}(t) + i W^{\perp}(t) \|_{H^1(\Gamma)} \leq A
 \left( \delta + \epsilon^{3/2} \right), \quad t \in [0,t_0].
\end{equation}
It makes sense to define $\delta = \mathcal{O}(\epsilon^{3/2})$ in
the bound (\ref{initial-data}) on the initial data, which we will adopt here.
By using the decomposition (\ref{r1}) and the bounds (\ref{bound-coefficients-apriori}) and (\ref{bound-omega-U-W-new})
in (\ref{bounds-parameters}), we get
\begin{equation}
\label{bounds-parameters-sharp}
| \dot{\theta} - \omega | \leq A \epsilon^2,
\quad | \dot{\omega} | \leq A \epsilon^{5/2},
\end{equation}
for an $\epsilon$-independent constant $A > 0$. By subtracting the first equation of system
(\ref{normal-form-time}) from the first equation of system (\ref{system-U-W}), we obtain
\begin{equation}
\label{equation-c}
\dot{c}_j - \dot{\gamma}_j = b_j - \beta_j + [F(c,b)]_j,
\end{equation}
where the vector $F(c,b) \in \mathbb{R}^{N-1}$ satisfies the estimate
\begin{equation}
\label{source-F}
\| F(c,b) \| \leq A \epsilon^{5/2},
\end{equation}
thanks to (\ref{RU-term-leading}), (\ref{r1}), (\ref{bound-omega-U-W-new}), and (\ref{bounds-parameters-sharp}).
By subtracting the second equation of system
(\ref{normal-form-time}) from the second equation of system (\ref{system-U-W}), we obtain
\begin{equation}
\label{equation-b}
\dot{b}_j - \dot{\beta}_j = p(p+1) (2p+1) \sum_{k=1}^{N-1} \sum_{n = 1}^{N-1}
\frac{\langle \Phi^{2p-1} U^{(k)} U^{(n)}, U^{(j)} \rangle_{L^2(\Gamma)}}{\langle W^{(j)},U^{(j)} \rangle_{L^2(\Gamma)}}
(c_k c_n - \gamma_k \gamma_n) + [G(c,b)]_j,
\end{equation}
where the vector $G(c,b) \in \mathbb{R}^{N-1}$ satisfies the estimate
\begin{equation}
\label{source-G}
\| G(c,b) \| \leq A \epsilon \nu(\epsilon),
\end{equation}
thanks to (\ref{RW-term-leading}), (\ref{r1}), (\ref{bound-omega-U-W-new}), and (\ref{bounds-parameters-sharp}),
where $\nu(\epsilon)$ is given by (\ref{nu}).

Let us assume than $\gamma(0) = c(0)$ and $\beta(0) = \beta(0)$. Integrating equations (\ref{equation-c}) and (\ref{equation-b})
over $t \in [0,t_0]$ with $t_0 \leq A \epsilon^{-1/2}$ in the region (\ref{bound-coefficients-apriori}), we obtain
\begin{equation}
\label{equation-c-integral}
\| c(t) - \gamma(t) \| \leq \int_0^t \| b(t') - \beta(t') \| dt' + A \epsilon^2
\end{equation}
and
\begin{equation}
\label{equation-b-integral}
\| b(t) - \beta(t) \| \leq A \epsilon \int_0^t \| c(t') - \gamma(t') \| dt' + A \epsilon^{1/2} \nu(\epsilon),
\end{equation}
for a generic $\epsilon$-independent constant $A > 0$. Gronwall's inequality for
$$
\| b(t) - \beta(t) \| + A \epsilon^{1/2} \| c(t) - \gamma(t) \|
$$
yields (\ref{bound-coefficients}).
\end{proof}

\begin{proof1}{\em of Theorem \ref{theorem-instability}.}
Let us consider the unstable solution $(\gamma,\beta)$ to the finite-dimensional system (\ref{normal-form-time})
according to Lemma \ref{lem-instability-Ham}. This solution belongs to the region (\ref{bound-coefficients-apriori}).
By Lemma \ref{lem-correction}, the correction terms satisfy (\ref{bound-coefficients}), hence
the solution $(c,b)$ to system (\ref{system-U-W}) also satisfies the bound (\ref{bound-coefficients-apriori})
over the time span $[0,t_0]$ with $t_0 = \mathcal{O}(\epsilon^{-1/2})$.

By Lemma \ref{lem-remainder-last} and the elementary continuation argument,
the components $\omega$, $U^{\perp}$, and $W^{\perp}$ satisfy the bound (\ref{bound-omega-U-W-new})
with $\delta = \mathcal{O}(\epsilon^{3/2})$, so that the solution to the NLS equation (\ref{eq1})
given by (\ref{orth-decomposition-time}) and (\ref{r1}) satisfies the bound (\ref{apriori-bound}) for $t \in [0,t_0]$.

Finally, the solution $\gamma$ to the finite-dimensional system (\ref{normal-form-time})
grows in time and reaches the boundary in the region (\ref{bound-coefficients-apriori})
by Lemma \ref{lem-instability-Ham}. The same is true for the full solution to the NLS equation (\ref{r1})
thanks to the bounds (\ref{bound-coefficients}) and (\ref{bound-omega-U-W-new}).
Hence, the solution starting with the initial data
satisfying the bound (\ref{initial-data}) with $\delta = \mathcal{O}(\epsilon^{3/2})$
reaches and crosses the boundary in (\ref{orbital-instab}) for some $t_0 = \mathcal{O}(\epsilon^{-1/2})$.
\end{proof1}

\end{document}